\documentclass[a4paper, 12pt]{amsart}

\usepackage{amssymb,amsmath}
\usepackage[dvips]{graphics}
\usepackage{graphicx}
\usepackage[english]{babel}
\usepackage[latin1]{inputenc}
\usepackage{comment}
\usepackage{color}

\newtheorem{theorem}{Theorem}[section]
\newtheorem{lemma}[theorem]{Lemma}

\newtheorem{corollary}[theorem]{Corollary}

\theoremstyle{definition}
\newtheorem{definition}[theorem]{Definition}

\newtheorem{remark}[theorem]{Remark}

\numberwithin{equation}{section}

\title[Measure density and extension of Besov and T--L functions]{Measure density and extension
of \\ Besov and Triebel--Lizorkin functions}

\author{Toni Heikkinen}
\author{Lizaveta Ihnatsyeva}
\author{Heli Tuominen*}

\thanks{\\T.H.: Department of Mathematics, P.O. Box 11100, FI-00076 Aalto University, Finland, toni.heikkinen@aalto.fi,\\
L.I.: Department of Mathematics and Statistics, P.O. Box 35, FI-40014 University of Jyv\"askyl\"a, 
Finland, lizaveta.ihnatsyeva@aalto.fi\\
H.T.: Department of Mathematics and Statistics, P.O. Box 35, FI-40014 University of Jyv\"askyl\"a, Finland,
heli.m.tuominen@jyu.fi, +358 40 805 4594, * the corresponding author}

\newcommand\rn{\mathbb R^n}
\newcommand\re{\mathbb R}
\newcommand\n{\mathbb N}
\newcommand\z{\mathbb Z}
\newcommand\D{\mathbb D}

\newcommand\ph{\varphi}
\newcommand\eps{\varepsilon}
\newcommand\M{\operatorname{\mathcal M}}
\newcommand\cB{\mathcal B}
\newcommand\cF{\mathcal F}
\newcommand\cH{\mathcal H}

\providecommand{\ch}[1]{\text{\raise 2pt \hbox{$\chi$}\kern-0.2pt}_{#1}}

\providecommand{\vint}[1]{\mathchoice
          {\mathop{\vrule width 5pt height 3 pt depth -2.5pt
                  \kern -9.5pt \kern 1pt\intop}\nolimits_{\kern -5pt{#1}}}%
          {\mathop{\vrule width 5pt height 3 pt depth -2.6pt
                  \kern -6pt \intop}\nolimits_{\kern -3pt{#1}}}%
          {\mathop{\vrule width 5pt height 3 pt depth -2.6pt
                  \kern -6pt \intop}\nolimits_{\kern -3pt{#1}}}%
          {\mathop{\vrule width 5pt height 3 pt depth -2.6pt
                  \kern -6pt \intop}\nolimits_{\kern -3pt{#1}}}}

\begin{document}

\begin{abstract}

We show that a domain is an extension domain for a Haj\l asz--Besov or for a Haj\l asz--Triebel--Lizorkin space if and only if
it satisfies a measure density condition.  We use a modification of the Whitney extension where integral averages  are 
replaced by median values, which allows us to handle also the case $0<p<1$. The necessity of the measure density condition 
is derived from embedding theorems; in the case of Haj\l asz--Besov spaces we apply an optimal Lorentz-type Sobolev 
embedding theorem which we prove using a new interpolation result. This interpolation theorem says that Haj\l asz--Besov 
spaces are intermediate spaces between $L^p$ and  Haj\l asz--Sobolev spaces.
Our results are proved in the setting of a metric measure space, but most of them are new even in the Euclidean setting,
for instance, we obtain a characterization of extension domains for classical Besov spaces 
$B^s_{p,q}$, $0<s<1$, $0<p<\infty$, $0<q\le\infty$, defined via the $L^p$-modulus of smoothness of a function.

\end{abstract}


\maketitle


\section{Introduction}

The restriction and extension problems for Besov spaces and Triebel--Lizorkin spaces in the setting of the Euclidean space
have been studied by several authors using different methods; see for example \cite{Mu}, \cite{BIN}, \cite{Ka}, \cite{Sh86}, 
\cite{Ry1}, \cite{Se}, \cite{DS}, \cite{Mi}, \cite{ShRn} and the references therein.
In particular, it is known that if $\Omega\subset\rn$ is a Lipschitz domain or an $(\eps,\delta)$-domain,
then there is a bounded extension operator from the classical Besov space $B^s_{p,q}(\Omega)$, defined via the $L^p$-
modulus of smoothness of a function, to $B^s_{p,q}(\rn)$, $0<s,p,q<\infty$; see \cite{Sh86} and \cite{DS}.
The analogous extension results hold for Triebel--Lizorkin spaces; see \cite{Se}, \cite{Mi} and \cite{T}.

Although the class of the $(\eps,\delta)$-domains, defined in \cite{J}, is rather wide, it does not cover all domains which admit 
an extension property for Besov spaces or for Triebel--Lizorkin spaces.
For example, by \cite[Thm 5.1]{Ry2}, some $d$-thick domains in $\rn$, measured with the $d$-dimensional Hausdorff content, 
are extension domains for certain Besov and Triebel--Lizorkin spaces.
It is also known that the trace spaces of Besov and Triebel--Lizorkin spaces to an $n$-regular
set $S\subset\rn$ can be intrinsically characterized \cite[Thm 1.3, Thm 1.6]{ShRn}, such $S$ admits an extension for Besov 
and Triebel--Lizorkin spaces defined in terms of local polynomial approximations.
See also \cite{Sh86} and \cite{JW} for the related results.

The connection of the $n$-regularity condition, or, in other words, a measure density condition,
and the extension property for Sobolev spaces is studied in \cite{ShRn} and in \cite{HKT_JFA}. 
By \cite[Thm 5]{HKT_JFA}, a domain $\Omega\subset\rn$ is an extension domain for $W^{k,p}$, $1<p<\infty$, $k\in\n$,  if and 
only if $\Omega$ satisfies the measure density condition and $W^{k,p}(\Omega)$ coincides with the Calderon-Sobolev space 
defined via sharp maximal functions.
For fractional Sobolev spaces $W^{s,p}(\Omega)$, $0<s<1$, $0<p<\infty$, which are special cases of Besov and Triebel--
Lizorkin spaces when $p=q$, the measure density condition characterizes extension domains by \cite[Thm 1.1]{Z}.
A natural question to ask is whether the same statement is true for Besov and Triebel--Lizorkin spaces within the full range of 
parameters $0<p<\infty$, $0<q\le\infty$. Moreover, if $0<s<1$, this question can be studied in a general setting of a metric 
measure space. The recent development of the theory of function spaces in metric measure spaces does
not only provide a uniform approach for characterizing smoothness function spaces on topological manifolds, fractals, graphs, 
and Carnot--Carath\'eodory spaces, but at the same time it gives a new point of view to the classical Besov spaces and 
Triebel--Lizorkin spaces on the Euclidean space.

Among several possible definitions of Besov and Triebel--Lizorkin spaces in the metric setting, the definition recently 
introduced in \cite{KYZ} appears to be very convenient for the study of extension problems.
This approach is based on Haj\l asz type pointwise inequalities; it leads to the classical Besov and Triebel--Lizorkin spaces in 
the setting of the Euclidean space and it gives a simple way to define these spaces on a measurable subset of $\rn$ or, more 
generally, on a metric measure space.

\begin{definition}
Let $(X,d)$ be a metric space equipped with a measure $\mu$.
A measurable set $S\subset X$ satisfies a {\em measure density condition}, if there exists a constant $c_{m}>0$ such that
\begin{equation}\label{measure density}
\mu(S\cap B(x,r))\ge c_{m}\mu(B(x,r))
\end{equation}
for all balls $B(x,r)$ with $x\in S$ and $0<r\le 1$.
\end{definition}

Note that in the literature sets satisfying condition \eqref{measure density} are sometimes called regular sets, see, for 
example, \cite{ShMetric}. If the measure $\mu$ is doubling, then the upper bound $1$ for the radius $r$ is not essential, and 
we can replace it by any number $0<R<\infty$. Roughly speaking, the measure density condition means that the set $S$ 
cannot be too thin near the boundary, in particular, by \cite[Lemma 2.1]{ShMetric}, it implies that $\mu(\overline S\setminus 
S)=0$.
In the Euclidean space, nontrivial examples of sets satisfying the measure density condition are Cantor-like sets such as 
Sierpi\'nski carpets of positive measure.
\medskip

Recall that if $\mathcal{A}$ is a quasi-Banach space of measurable functions and $S\subset X$,
an operator $E\colon \mathcal{A}(S)\to \mathcal{A}(X)$ such that $Eu\vert_S=u$, for all $u\in \mathcal{A}(S)$,
is called an {\em extension operator}.
A domain $\Omega\subset X$ is an $\mathcal{A}$-{\em extension domain} if there is a bounded extension operator
$E\colon \mathcal{A}(\Omega)\to \mathcal{A}(X)$.

In the metric setting, a connection between the measure density condition and the extension property for Sobolev spaces has 
been studied in \cite{HKT_Revista} and in \cite{ShMetric}.
By \cite[Thm 6]{HKT_Revista}, \cite[Thm 1.3]{ShMetric}, the measure density condition for a set $S$ implies the existence of a 
bounded, linear extension operator on the Ha\l asz--Sobolev space $M^{1,p}(S)$, for all $1\le p<\infty$.
In a geodesic, $Q$-regular metric measure space, the measure density condition characterizes extension domains for 
$M^{1,p}$, $1\le p<\infty$, see \cite{HKT_Revista}.

Our first main result is an extension theorem for Haj\l asz--Triebel--Lizorkin spaces $M^s_{p,q}$ and for Haj\l asz--Besov 
spaces $N^s_{p,q}$, see Section \ref{preliminaries} for the definitions.

\begin{theorem}\label{m extension}
Let $X$ be a metric measure space with a doubling measure $\mu$ and let $S\subset X$ be a measurable set.
If $S$ satisfies measure density condition \eqref{measure density},
then there is a bounded extension operator $E\colon M^s_{p,q}(S)\to M^s_{p,q}(X)$, for all $0<p<\infty$, $0<q\le\infty$, 
$0<s<1$.
The operator norm of $E$ depends on $c_m$, $p$, $q$, $s$ and on the doubling constant of $\mu$.
An analogous extension result holds for Haj\l asz--Besov spaces $N^s_{p,q}$.
\end{theorem}

As in the corresponding results for the Haj\l asz--Sobolev spaces $M^{1,p}$ in \cite{HKT_Revista}
and for the fractional Sobolev spaces $W^{s,p}$, $0<s<1$, $0<p<\infty$, in \cite{Z}, the extension is independent of the 
parameters of a function space.
In general, our extension operator is not linear. This is due to the use of a modified Whitney type extension where integral 
averages are replaced by medians; similar modification was previously used in \cite{Z}. But if $p>Q/(Q+s)$, where $Q$ is the 
doubling dimension of the space, an extension operator in 
Theorem \ref{m extension} can be chosen linear by employing the classical construction with integral averages.

In the Euclidean case, $N^s_{p,q}(\rn)=B^s_{p,q}(\rn)$ and $M^s_{p,q}(\rn)=F^s_{p,q}(\rn)$ for all $0<p<\infty$, $0<q\le\infty$, 
$0<s<1$, 
where $B^s_{p,q}(\rn)$ and $F^s_{p,q}(\rn)$ are Besov spaces and Triebel--Lizorkin spaces defined via an
$L^p$-modulus of smoothness, \cite{GKZ}; recall that the Fourier analytic approach gives the same spaces when $p>n/(n+s)$ 
in the Besov case and when $p,q>n/(n+s)$ in the Triebel--Lizorkin case. Theorem \ref{m extension}, in particular, shows that 
the trace spaces of the classical Besov and Triebel--Lizorkin spaces on regular sets can be characterized in terms of pointwise 
inequalities. Indeed, it follows that $B^s_{p,q}(\rn)|_{S}=N^s_{p,q}(S)$ and $F^s_{p,q}(\rn)|_{S}=M^s_{p,q}(S)$ with equivalent 
norms.
\medskip

The following statement, which is a combination of Theorem \ref{m extension} and Theorem \ref{measure density from 
extension}, is our second main result.

\begin{theorem}\label{m extension measure}
Let $X$ be a $Q$-regular, geodesic metric measure space and let $\Omega\subset X$ be a domain.
The following conditions are equivalent:
\begin{enumerate}
\item $\Omega$ satisfies measure density condition \eqref{measure density};

\item $\Omega$ is an $M^s_{p,q}$-extension domain for all $0<s<1$, $0<p<\infty$ and $0<q\le\infty$;

\item $\Omega$ is an $M^s_{p,q}$-extension domain for some values of parameters
$0<s<1$, $0<p<\infty$ and $0<q\le\infty$.

\item $\Omega$ is an $N^s_{p,q}$-extension domain for all $0<s<1$, $0<p<\infty$ and $0<q\le\infty$;

\item $\Omega$ is an $N^s_{p,q}$-extension domain for some values of parameters
$0<s<1$, $0<p<\infty$ and $0<q\le\infty$.
\end{enumerate}
\end{theorem}

To our knowledge, the fact that an extension domain for the Besov space, or for the Triebel--Lizorkin space, necessarily 
satisfies
the measure density condition is also new in the Euclidean setting; except for the special case $p=q$ which was earlier proved 
in 
\cite{Z}.
Let us also mention that the assumption on $X$ to be geodesic is used only to guarantee the property that the boundary of 
each metric ball has zero measure.

In order to show that extension domains satisfy the measure density property, we need a suitable Sobolev type embedding 
theorem. For Haj\l asz--Triebel--Lizorkin spaces such an embedding is easy to get, since they are subspaces of fractional Haj\l 
asz--Sobolev spaces.
To obtain an embedding theorem for Haj\l asz--Besov spaces, we show in Theorem \ref{thm: interpolation} that Haj\l asz--
Besov spaces are interpolation spaces between $L^p$ and Haj\l asz--Sobolev spaces, that is,
\[
N^s_{p,q}(X)=(L^p(X),M^{1,p}(X))_{s,q},
\]
for $0<s<1$, $0<p<\infty$ and $0<q\le\infty$.
We found this result of independent interest; in case of $p>1$, $q\ge 1$ it was earlier obtained in \cite{GKS} under an 
additional assumption that the underlying space supports a weak $(1,p)$-Poincar\'e inequality.
\medskip

We close the paper with an application of Theorem \ref{m extension measure}
to Besov and Triebel--Lizorkin spaces defined in the Euclidean space. 
In particular, we obtain the following result for the classical Besov spaces $B^{s}_{p,q}$, $0<s<1$, $0<p<\infty$ and $0<q\le
\infty$, defined via the $L^p$-modulus of smoothness.

\begin{theorem}\label{ExtensionEuclideanCase}
Let $\Omega\subset\rn$ be a domain. The following conditions are equivalent:
\begin{enumerate}
\item $\Omega$ satisfies measure density condition \eqref{measure density};

\item $\Omega$ is a $B^s_{p,q}$-extension domain for all $0<s<1$, $0<p<\infty$ and $0<q\le\infty$;

\item $\Omega$ is a $B^s_{p,q}$-extension domain for some $0<s<1$, $0<p<\infty$ and $0<q\le\infty$;

\end{enumerate}
\end{theorem}

Extension problems are closely related to the question of intrinsic characterization of spaces of fractional order of smoothness 
on subsets $S\subset\rn$.
The obtained results shows that if $S$ satisfies the measure density condition, then the space $B^s_{p,q}(S)$, $0<s<1$, 
$0<p<\infty$, $0<q\le\infty$, can be defined via the $L^p$-modulus of smoothness, via pointwise inequalities, in terms of an 
atomic decomposition, see \cite{Tr3} for the details on the last-mentioned approach; all these definitions would lead to the 
same space of functions which is the trace space of the classical Besov space $B^s_{p,q}(\rn)$.

We also give an analogue of Theorem \ref{ExtensionEuclideanCase} for certain Triebel--Lizorkin spaces, see Theorem \ref{ext 
TL measure density}.
Note that there are several approaches to define Triebel--Lizorkin spaces on domains, which, in general, give different spaces.
In Theorem \ref{ext TL measure density}, we use a definition in the spirit of the classical definition of $F^{s}_{p,q}(\rn)$ via 
differences; it describes, for example, the trace space of $F^{s}_{p,q}(\rn)$ to a regular subset of the Euclidean space. Another 
version of Triebel--Lizorkin type spaces on domains was introduced in \cite{Se} and \cite{Mi}; a characterization of extension 
domains for these spaces can be given similarly to the one for the Sobolev spaces in \cite{HKT_JFA}, see 
Theorem \ref{ext Cspq spaces measure density} and Remark \ref{rm:IntrinsicDefinitionsTriebel-Lizorkin}.

\medskip
The paper is organized as follows.
In Section \ref{preliminaries}, we introduce the notation and the standard assumptions used in the paper and give the 
definitions of Haj\l asz--Besov spaces and Haj\l asz--Triebel--Lizorkin spaces.
In Section \ref{section: lemmas}, we present some auxiliary lemmas needed in the proof of our extension results.
In Section \ref{section: interpolation}, we prove interpolation and embedding theorems for Besov spaces.
Section \ref{section: extensions} is devoted to the proof of Theorem \ref{m extension}.
In Section \ref{section: measure density from extension}, we show that Haj\l asz--Besov and  Haj\l asz--Triebel--Lizorkin 
extension domains satisfy measure density property.
In the last Section \ref{Rn}, we discuss the Euclidean case.

\section{Notation and preliminaries}\label{preliminaries}
We assume that $X=(X, d,\mu)$ is a metric measure space equipped with a metric $d$ and a Borel regular,
doubling outer measure $\mu$, for which the measure of every ball is positive and finite.
The {\em doubling} property means that there exists a fixed constant $c_D>0$, called  {\em the doubling constant}, such that
\begin{equation}\label{doubling measure}
\mu(B(x,2r))\le c_D\mu(B(x,r))
\end{equation}
for every ball $B(x,r)=\{y\in X:d(y,x)<r\}$.

The doubling condition gives an upper bound for the dimension of $X$ since
it implies that there is a constant $C=C(c_D)>0$ such that for $Q=\log_2c_D$,
\begin{equation}\label{doubling dimension}
\frac{\mu(B(y,r))}{\mu(B(x,R))}\ge C\Big(\frac rR\Big)^Q
\end{equation}
for every $0<r\le R$ and $y\in B(x,R)$.

As a special case of doubling spaces we consider $Q$-regular spaces.
The space $X$ is  {\em $Q$-regular}, $Q\ge1$, if there is a constant $c_Q\ge1$ such that
\begin{equation}\label{Q-regular}
c_Q^{-1}r^Q\le \mu(B(x,r))\le c_Qr^Q
\end{equation}
for each $x\in X$, and for all $0<r\le\operatorname {diam}X.$ Here $\operatorname{diam} X$ is the diameter of $X$. When we 
assume $X$ to be doubling, then $Q$ refers to \eqref{doubling dimension}, and if $X$ is $Q$-regular, then $Q$ comes from 
\eqref{Q-regular}.

A metric space $X$ is  {\em geodesic} if every two points $x,y\in X$ can be joined by a curve whose length equals $d(x,y)$.

By saying that a measurable function $u\colon X\to[-\infty,\infty]$ is  {\em locally integrable}, we mean that is integrable on 
balls.
Similarly, the class of functions that belong to $L^p(B)$, $p>0$, in all balls $B$, is denoted by $L^p_{\text{loc}}(X)$.
The {\em integral average} of a locally integrable function $u$ over a measurable set $A$ with $0<\mu(A)<\infty$ is
\[
u_A=\vint{A}u\,d\mu=\frac1{\mu(A)}\int_Au\,d\mu.
\]
The  {\em Hardy--Littlewood maximal function} of a locally integrable function $u$ is
\[
\M u(x)=\sup_{0<r<\infty}\;\vint{B(x,r)}|u|\,d\mu.
\]
By $\ch{E}$, we denote the characteristic function of a set $E\subset X$, and by $\|u\|_\infty$, the $L^\infty$-norm of $u$.
The Lebesgue measure of a measurable set $A\subset\rn$ is denoted by $|A|$.
In general, $C$ is a positive constant whose value is not necessarily the same at each occurrence.
When we want to stress that $C$ depends on the other constants or parameters $a,b,\dots$, we write $C=C(a,b,\dots)$.
If there is a positive constant $C_1$ such that $C_1^{-1}A\le B\le C_1A$, we say that $A$ and $B$ are comparable, and write 
$A\approx B$.

\subsection{Haj\l asz--Besov and Haj\l asz--Triebel--Lizorkin spaces}\label{section TL}
Besov and Triebel--Lizorkin spaces are certain generalizations of Sobo\-lev spaces to the case of fractional order of 
smoothness.
There are several ways to define these spaces in the Euclidean setting and spaces of Besov type and of Triebel--Lizorkin type 
in the setting of a metric space equipped with a doubling measure. For various definitions in a metric measure setting, see 
\cite{GKS}, \cite{GKZ}, \cite{HMY}, \cite{KYZ}, \cite{MY}, \cite{SYY}, \cite{YZ} and the references therein.
In this paper, we mainly use the approach based on pointwise inequalities, introduced in \cite{KYZ}.
An advantage of the pointwise definition is that it provides a simple way to intrinsically define function spaces on subsets.

\begin{definition}
Let $S\subset X$ be a measurable set and let $0<s<\infty$.
A sequence of nonnegative measurable functions $(g_k)_{k\in\z}$ is a  {\em fractional $s$-gradient} of a measurable function
$u\colon S\to[-\infty,\infty]$ in $S$, if there exists a set $E\subset S$ with $\mu(E)=0$ such that
\[
|u(x)-u(y)|\le d(x,y)^s(g_k(x)+g_k(y))
\]
for all $k\in\z$ and all $x,y\in S\setminus E$ satisfying $2^{-k-1}\le d(x,y)<2^{-k}$.
The collection of all fractional $s$-gradients of $u$ is denoted by $\D^s(u)$.
\end{definition}

For $0<p,q\le\infty$ and a sequence $\vec f=(f_k)_{k\in\z}$ of measurable
functions, we define
\[
\big\|(f_k)_{k\in\z}\big\|_{L^p(S,\,l^q)}
=\big\|\|(f_k)_{k\in\z}\|_{l^q}\big\|_{L^p(S)}
\]
and
\[
\big\|(f_k)_{k\in\z}\big\|_{l^q(L^p(S))}
=\big\|\big(\|f_k\|_{L^p(S)}\big)_{k\in\z}\big\|_{l^q},
\]
where
\[
\big\|(f_k)_{k\in\z}\big\|_{l^{q}}
=
\begin{cases}
\big(\sum_{k\in\z}|f_{k}|^{q}\big)^{1/q},&\quad\text{when }0<q<\infty, \\
\;\sup_{k\in\z}|f_{k}|,&\quad\text{when }q=\infty.
\end{cases}
\]

\begin{definition}
Let $S\subset X$ be a measurable set. Let $0<s<\infty$ and let $0<p,q\le\infty$.
The  {\em homogeneous Haj\l asz--Triebel--Lizorkin space} $\dot M_{p,q}^s(S)$ consists of measurable functions
$u\colon S\to[-\infty,\infty]$, for which the (semi)norm
\[
\|u\|_{\dot M_{p,q}^s(S)}
=\inf_{\vec{g}\in\D^s(u)}\|\vec{g}\|_{L^p(S,\,l^q)}
\]
is finite. The  {\em (inhomogeneous) Haj\l asz--Triebel--Lizorkin space} $M_{p,q}^s(S)$ is $\dot M_{p,q}^s(S)\cap L^p(S)$
equipped with the norm
\[
\|u\|_{M_{p,q}^s(S)}=\|u\|_{L^p(S)}+\|u\|_{\dot M_{p,q}^s(S)}.
\]
Similarly, the  {\em homogeneous Haj\l asz--Besov space} $\dot N_{p,q}^s(S)$ consists of measurable functions
$u\colon S\to[-\infty,\infty]$, for which
\[
\|u\|_{\dot N_{p,q}^s(S)}=\inf_{(g_k)\in\D^s(u)}\|(g_k)\|_{l^q(L^p(S))}
\]
is finite, and the {\em Haj\l asz--Besov space} $N_{p,q}^s(S)$ is $\dot N_{p,q}^s(S)\cap L^p(S)$
equipped with the norm
\[
\|u\|_{N_{p,q}^s(S)}=\|u\|_{L^p(S)}+\|u\|_{\dot N_{p,q}^s(S)}.
\]
\end{definition}
When $0<p<1$, the (semi)norms defined above are actually quasi-(semi)norms, but for simplicity we call them, as well as 
other quasi-seminorms in this paper, just norms.

\begin{remark}\label{norm sum N}
Observe that for inhomogeneous Haj\l asz--Triebel--Lizorkin and Haj\l asz--Besov spaces the norms defined above are 
equivalent to
\[
\|u\|_{L^p(S)}+\inf_{\vec{g}\in\D^s(u)}\|(g_k)_{k\in\n}\|_{L^p(S,\,l^q)}\quad \text{and}\quad
\|u\|_{L^p(S)}+\inf_{\vec{g}\in\D^s(u)}\|(g_k)_{k\in\n}\|_{l^q(S,\,L^p)}
\]
respectively, that is, it is enough to take into account only the coordinates of $\vec{g}$ with positive indices.
Indeed, if $x,y\in S\setminus E$ and $2^{-k-1}\le d(x,y)<2^{-k}$ with $k\le 0$, then
\[
|u(x)-u(y)|
\le |u(x)|+|u(y)|
\le 2^{(k+1)s} d(x,y)^s(|u(x)|+|u(y)|).
\]
Hence, if $(g_k)_{k\in\z}\in \D^s(u)$, then $(g'_k)_{k\in\z}$, where $g'_k=g_k$ for $k>0$ and $g'_k=2^{(k+1)s}|u|$ for $k\le 0$, 
belongs to $\D^s(u)$.
Calculating the norm, for example, for Haj\l asz--Triebel--Lizorkin space, we obtain that
\[
\begin{split}
\|\vec{g'}\|_{L^p(S,\,l^q)}
&\le C\big(\|(g'_k)_{k\in\n}\|_{L^p(S,\,l^q)}+\|(g'_k)_{k\le 0}\|_{L^p(S,\,l^q)}\big)\\
&= C\|(g_k)_{k\in\n}\|_{L^p(S,\,l^q)}+ C\|u\|_{L^p(S)}\Big(\sum_{k=-\infty}^{0}2^{(k+1)sq}\Big)^{1/q},
\end{split}
\]
where the constants $C$ depend on $p$ and $q$ only. 
This implies that
\[
\inf_{\vec{g}\in\D^s(u)}\|\vec{g}\|_{L^p(S,\,l^q)}
\le C\big(\inf_{\vec{g}\in\D^s(u)}\|(g_k)_{k\in\n}\|_{L^p(S,\,l^q)}+\|u\|_{L^p(S)}\big).
\]
\end{remark}

If $X$ supports a weak $(1,p)$-Poincar\'e inequality with $p\in (1,\infty)$, then for all $q\in(0,\infty)$,
the spaces $M^1_{p,q}(X)$ and $N^1_{p,q}(X)$ are trivial, that is, they contain
only constant functions, see \cite[Thm 4.1]{GKZ}.
\medskip

The definitions formulated above are, in particular, motivated by the Haj\l asz's approach to the definition of Sobolev spaces
$M^{1,p}(X)$ on a metric measure space; see \cite{H} and \cite{H2}.
The fractional spaces $M^{s,p}(X)$ were introduced in \cite{Y}, and were studied, for example, in \cite{Hu} and \cite{HKT}.

\begin{definition}
Let $S\subset X$ be a measurable set. Let $s\ge 0$ and let $0<p<\infty$.
A nonnegative measurable function $g$ is an  {\em $s$-gradient} of a measurable function $u$ in $S$ if there exists a set
$E\subset S$ with $\mu(E)=0$ such that for all $x,y\in S\setminus E$,
\begin{equation}\label{eq: gradient}
|u(x)-u(y)|\le d(x,y)^s(g(x)+g(y)).
\end{equation}
The collection of all $s$-gradients of $u$ is denoted by $\mathcal{D}^s(u)$ and the $1$-gradients shortly by $\mathcal{D}(u)$.
The {\em homogeneous Haj\l asz space} $\dot{M}^{s,p}(S)$ consists of measurable functions $u$ for which
\[
\|u\|_{\dot M^{s,p}(S)}=\inf_{g\in\mathcal{D}^s(u)}\|g\|_{L^p(S)}
\]
is finite.
The {\em Haj\l asz space} $M^{s,p}(S)$ is $\dot M^{s,p}(S)\cap L^p(S)$  equipped with the norm
\[
\|u\|_{M^{s,p}(S)}=\|u\|_{L^p(S)}+\|u\|_{\dot M^{s,p}(S)}.
\]
\end{definition}
Recall that for $p>1$, $M^{1,p}(\rn)=W^{1,p}(\rn)$ \cite{H}, whereas for $n/(n+1)<p\le 1$,  $M^{1,p}(\rn)$ coincides with the 
Hardy--Sobolev space $H^{1,p}(\rn)$ \cite[Thm 1]{KS}.
Notice also that $M^{0,p}(X)=L^p(X)$ and that $M^{s,p}(X)$ coincides with
the Haj\l asz--Triebel--Lizorkin space $M_{p,\infty}^s(X)$, see \cite[Prop.\ 2.1]{KYZ} for a simple proof of this fact.

\subsection {On different definitions of Besov and Triebel--Lizorkin spaces}\label{section: different definitions}
In the Euclidean setting the most common ways to define Besov and Triebel--Lizorkin spaces, 
via the $L^p$-modulus of smoothness (differences) and by the Fourier analytic approach, lead to the same spaces of 
functions 
with comparable norms when $p>n/(n+s)$ in the Besov case and when $p,q>n/(n+s)$ in the Triebel--Lizorkin case.
See, for example, \cite[Chapter 2.5]{Tr} and \cite{HS}.

The space $M^s_{p,q}(\rn)$ given by the metric definition coincides with Triebel--Lizor\-kin space ${\bf F}^s_{p,q}(\rn)$, 
defined via the Fourier analytic approach, when $s\in (0,1)$, $p\in(n/(n+s),\infty)$ and $q\in(n/(n+s),\infty]$, and
$M^1_{p,\infty}(\rn)=M^{1,p}(\rn)={\bf F}^{1}_{p,2}(\rn)$, when $p\in(n/(n+1),\infty)$.
Similarly, $N_{p,q}^s(\rn)$ coincides with Besov space ${\bf B}^s_{p,q}(\rn)$ for $s\in (0,1)$, $p\in(n/(n+s),\infty)$ and $q
\in(0,\infty]$, see \cite[Thm 1.2 and Remark 3.3]{KYZ}. For the definitions of ${\bf F}^s_{p,q}(\rn)$ and
${\bf B}^s_{p,q}(\rn)$, we refer to \cite{Tr}, \cite{Tr2}, \cite[Section 3]{KYZ}.

\subsection{Modulus of smoothness and Besov spaces}\label{section: mod of smooth}
In addition to the definition based on pointwise inequalities, we will sometimes use a generalization to the metric setting of
the classical definition of the Besov spaces via the $L^p$-modulus of smoothness; this general version was introduced in 
\cite{GKS}.
\smallskip

Recall that the {\em $L^p$-modulus of smoothness} of a function $u\in L^p(\rn)$ is
\begin{equation}\label{omega smoothness}
\omega(u,t)_p=\sup_{|h|\le t}\|\Delta_h(u,\cdot)\|_{L^p(\rn)},
\end{equation}
where $t>0$ and $\Delta_h(u,x)=u(x+h)-u(x)$.
For $0<s<\infty$ and $0<p,q<\infty$, the Besov space $B_{p,q}^s(\rn)$ consists of functions $u\in L^p(\rn)$ for which
\[
\|u\|_{B^s_{p,q}(\rn)}
=\|u\|_{L^p(\rn)}+\bigg(\int_0^1\big(t^{-s}\omega(u,t)_p\big)^{q}\frac{dt}{t}\bigg)^{1/q}
\]
is finite (with the usual modifications when $p=\infty$ or $q=\infty$). Note that
the integral over the interval $(0,1)$ can be replaced by the integral over $(0,\infty)$, since $\omega(u,t)_p\le C\|u\|_{L^p(\rn)}
$.

Following \cite{GKS} and \cite{GKZ}, we define a modulus of smoothness which does not rely on the group structure of the 
underlying space and which, for a function $u\in L^p(\rn)$, is comparable with $\omega(u,t)_p$.

\begin{definition}\label{B global}
Let $t>0$, $0<s<\infty$ and $0<p,q<\infty$.  Let
\begin{equation}\label{eput}
E_p(u,t)=\Big(\int_X\vint{B(x,t)}|u(x)-u(y)|^p\,d\mu(y)d\mu(x)\Big)^{1/p}.
\end{equation}
The {\em homogeneous Besov space} $\dot{\cB}_{p,q}^s(X)$ consists of functions
$u\in L^p_{\text{loc}}(X)$ for which
\[
\|u\|_{\dot{\cB}^s_{p,q}(X)}
=\bigg(\int_0^\infty\big(t^{-s}E_p(u,t)\big)^{q}\frac{dt}{t}\bigg)^{1/q}
\]
is finite (with the usual modification when  $q=\infty$).
The {\em Besov space} ${\cB}_{p,q}^s(X)$ is $\dot{\cB}_{p,q}^s(X)\cap L^p(X)$ with the norm
\[
\|u\|_{{\cB}^s_{p,q}(X)}=\|u\|_{L^p(X)}+\|u\|_{\dot{\cB}^s_{p,q}(X)}.
\]
\end{definition}
By the comparability of $\omega(u,t)_p$ and $E_p(u,t)$, the space ${\cB}^s_{p,q}(\rn)$ coincides with the classical space 
$B^s_{p,q}(\rn)$.
By \cite[Thm 1.2]{GKZ}, $\dot N^s_{p,q}(X)=\dot{\cB}_{p,q}^s(X)$ for all $0<s<\infty$ and $0<p,q\le\infty$, and
\begin{equation}\label{equivalence of Besov norms}
\|u\|_{\dot N^s_{p,q}(X)}
\approx\|u\|_{\dot{\cB}^s_{p,q}(X)}.
\end{equation}
As above, the integral over the interval $(0,\infty)$ in the norm $\|u\|_{\cB^s_{p,q}(X)}$ can be replaced by the integral over 
$(0,1)$.

It also follows by the results in \cite{GKZ} that, for $0<s<1$, $0<p,q\le\infty$, the Haj\l asz--Triebel--Lizorkin space $M^s_{p,q}
(\rn)$ coincides
with the classical Triebel--Lizorkin space ${F}^s_{p,q}(\rn)$ defined using differences.
This space consists of functions $u\in L^p(\rn)$, for which the norm
\[
\|u\|_{{F}^s_{p,q}(\rn)}
=\|u\|_{L^p(\rn)}+\|g\|_{L^p(\rn)},
\]
where
\[
\begin{split}
g(x)
=&\bigg(\int_0^1\bigg(t^{-s}\Big(\,\vint{B(0,t)}
|u(x+h)-u(x)|^r\,dh\Big)^{1/r}\bigg)^q\,\frac{dt}{t}\bigg)^{1/q}\\
\end{split}
\]
and $0<r<\min\{p,q\}$, is finite.

\section{Lemmas}\label{section: lemmas}
This section contains lemmas needed in the proofs of the extension results.

Below we will frequently use the following simple inequality, which holds whenever $a_i\ge 0$ for all $i\in\z$ and $0<p\le 1$,
\begin{equation}\label{a sum}
\big(\sum _{i\in\z}a_i\big)^p\le\sum_{i\in\z} a_i^p.
\end{equation}

The first lemma is used to estimate the norms of fractional gradients.
\begin{lemma}\label{summing lemma}
Let $1<a<\infty$, $0<b<\infty$ and $c_k\ge 0$, $k\in\z$. There is a constant $C=C(a,b)$ such that
\[
\sum_{k\in\z}\Big(\sum_{j\in\z}a^{-|j-k|}c_j\Big)^b\le C\sum_{j\in\z}c_j^b.
\]
\end{lemma}
\begin{proof}
If $b\ge 1$, then the H\"older's inequality for series implies that
\[
\Big(\sum_{j\in\z}a^{-|j-k|}c_j\Big)^b\le C\sum_{j\in\z}a^{-|j-k|}c_j^b.
\]
If $0<b< 1$, then, by \eqref{a sum},
\[
\Big(\sum_{j\in\z}a^{-|j-k|}c_j\Big)^b\le \sum_{j\in\z}a^{-b|j-k|}c_j^b.
\]
Thus, denoting $\tilde b=\min\{b,1\}$, we obtain
\[
\sum_{k\in\z}\Big(\sum_{j\in\z}a^{-|j-k|}c_j\Big)^b
\le C\sum_{k\in\z}\sum_{j\in\z}a^{-\tilde b|j-k|}c_j^b
\le C\sum_{j\in\z} c_j^b\sum_{k\in\z} a^{-\tilde b|j-k|},
\le C\sum_{j\in\z} c_j^b,
\]
which proves the claim.
\end{proof}

Next we recall the Poincar\'e type inequalities which are valid for functions and fractional gradients,
give a definition of median values and list some of their properties and obtain certain norm estimates for Lipschitz functions.

\subsection{Poincar\'e type inequalities}
The definition of the fractional $s$-gradient implies the validity of some Sobolev--Poincar\'e type inequalities.
A similar reasoning as in the proof of \cite[Lemma 2.1]{KYZ} in $\rn$ gives our first inequality.

\begin{lemma}\label{lemma: Poincare}
Let $0<s<\infty$. Let $u$ be a locally integrable function and let $(g_j)\in \D^s(u)$.
Then, for every $x\in X$ and $k\in\z$,
\begin{equation}\label{eq: Poincare}
\begin{split}
\inf_{c\in \re}\, \vint{B(x,2^{-k})}|u-c|\,d\mu
\le C2^{-ks}\sum_{j= k-3}^k\;\vint{B(x,2^{-k+2})}g_j\,d\mu.
\end{split}
\end{equation}
\end{lemma}

Note that we will apply Lemma \ref{lemma: Poincare} for functions in $N^{s}_{p,q}(X)$ with $sp>Q$ and with $sp=Q$, and for 
these values of parameters functions in $N^{s}_{p,q}(X)$ are locally integrable.

\begin{lemma}[\cite{GKZ}, Lemma 2.1]\label{lemma:Sobolev-Poincare 2}
Let $0<s<\infty$ and $0<t<Q/s$.
Then for every $\eps$ and $\eps'$ with $0<\eps<\eps'<s$, there exists
a constant $C>0$ such that for all measurable functions $u$ with $(g_j)\in \D^s(u)$, $x\in X$
and $k\in\z$,
\begin{equation}\label{Sobolev-Poincare 2}
\inf_{c\in\re}\Big(\,\vint{B(x,2^{-k})}|u(y)-c|^{t^*(\eps)}\,d\mu(y)\Big)^{1/t^*(\eps)}
\le C2^{-k\eps'}\sum_{j\ge k-2}2^{-j(s-\eps')}
\Big(\,\vint{B(x,2^{-k+1})}g_j^{t}\,d\mu\Big)^{1/t},
\end{equation}
where $t^*(\eps)=Qt/(Q-\eps t)$.
\end{lemma}
If $u$ is locally integrable, $(g_j)\in \D^s(u)$ and $0<\eps<\eps'<s<\infty$, then \eqref{Sobolev-Poincare 2}
with $t=Q/(Q+\eps)$ and the H\"older's inequality imply that for $p\ge Q/(Q+\eps)$,
\begin{equation}\label{Poincare 2}
\begin{split}
\vint{B(x,2^{-k})}|u-u_{B(x,2^{-k})}|\,d\mu
\le C2^{-k\eps'}\sum_{j\ge k-2}2^{-j(s-\eps')}
\Big(\,\vint{B(x,2^{-k+1})}g_j^{p}\,d\mu\Big)^{1/p}.
\end{split}
\end{equation}

While working with the Haj\l asz--Triebel-Lizorkin spaces $M^{s}_{p,q}(X)$ we often use an embedding of these spaces into 
the
space $M^{s,p}(X)$ and employ the following Sobolev-Poincar\'e inequality for $s$-gradients.

\begin{lemma}[\cite{GKZ}, Lemma 2.2]\label{lemma:Sobolev-Poincare 1}
Let $0<s<\infty$ and $0<t<Q/s$.
There exists a constant $C>0$ such that for all measurable functions $u$ with $g\in\mathcal D^s(u)$, $x\in X$
and $r>0$,
\begin{equation}\label{Sobolev-Poincare 1}
\inf_{c\in\re}\Big(\,\vint{B(x,r)}|u(y)-c|^{t^*(s)}\,d\mu(y)\Big)^{1/t^*(s)}\\
\le Cr^s
\Big(\,\vint{B(x,2r)}g^{t}\,d\mu\Big)^{1/t},
\end{equation}
where $t^*(s)=Qt/(Q-st)$.
\end{lemma}
For $s=1$ inequality \eqref{Sobolev-Poincare 1} is given by \cite[Thm 8.7]{H2}, as well as
for $s\in(0,1)$, since in this case $d^s$ is a distance in $X$.

\subsection{Median values}
Using integral averages of a function is a standard technique in construction an extension operator for a locally integrable 
function.
Since we are dealing with the $L^p$-integrable functions, possibly with $0<p<1$, it is convenient to replace in the argument
the integral averages by the median values, as for example in \cite{Z}. 
This allows to handle in the same way spaces of functions with the integrability parameter $0<p<\infty$;
a certain disadvantage of this uniform treatment is that the resulting extension operator appears to be non-linear.

\begin{definition}
The median value of a measurable function $u$ on a set $A\subset X$ is
\begin{equation}\label{median}
m_{u}(A)=\max_{ a\in\re}\bigg\{\mu\big(\{x\in A:u(x)<a\}\big)\le\frac{\mu(A)}2\bigg\}.
\end{equation}
\end{definition}

The following properties of medians justify their role of counterparts for the integral averages in the context.

\begin{lemma}[\cite{Z}, Lemma 2.2; \cite{GKZ}, (2.4)]\label{median lemma}
Let $0<\eta\le 1$ and $u\in L_{\text{loc}}^{\eta}(X)$. Then
\begin{equation}\label{median ie}
|m_{u}(B)-c|
\le\Big(2\vint{B}|u-c|^{\eta}\,d\mu\Big)^{1/\eta}.
\end{equation}
for all balls $B$ and all $c\in\re$. Moreover,
\begin{equation}\label{median limit}
u(x)=\lim_{r\to 0}m_{u}(B(x,r))
\end{equation}
at every Lebesgue point $x\in X$.
\end{lemma}

\begin{remark}\label{median F}
Property \eqref{median limit} follows from \eqref{median ie} by the Lebesgue differentiation theorem.
The proof of \cite[Lemma 2.2]{Z} shows that inequality \eqref{median ie} holds for all measurable sets
$E$ with positive and finite measure. In particular, \eqref{median ie} holds for every set $B\cap S$,
where $S$ satisfies measure density condition \eqref{measure density} and $B$ is a ball centered at $S$.
This, together with the measure density condition and the Lebesgue differentiation theorem, implies that,
\[
u(x)=\lim_{r\to 0}m_{u}(B(x,r)\cap S),
\]
for almost all $x\in S$.
\end{remark}

By combining \eqref{median ie}  and Lemma \ref{lemma:Sobolev-Poincare 2},
we obtain the following result, which is frequently used in the proof of Theorem \ref{m extension}.

\begin{lemma}\label{median difference}
Let $0<t<\infty$ and $0<\eps'<s<1$.
Let $k\in\z$, $x\in X$ and let $B$ be a ball such that $B\subset B(x,2^{-k})$  and $\mu(B)\approx \mu(B(x,2^{-k}))$.
Then there exists a constant $C>0$ such that for all measurable functions $u$ with $(g_j)\in \D^s(u)$,
\begin{equation}\label{estimate for medians}
|m_u(B)-m_u(B(x,2^{-k}))|
\le C2^{-k\eps'}\sum_{j\ge k-2}2^{-j(s-\eps')}\Big(\,\vint{B(x,2^{-k+1})}g_j^{t}\,d\mu\Big)^{1/t}.
\end{equation}
\end{lemma}

\begin{proof}
Let $0<\eps<\eps'$ and let $\eta=t^*(\eps)=Qt/(Q-\eps t)$ if $t\le Q/(Q+\eps)$, and $\eta=1$ otherwise.
Using \eqref{Sobolev-Poincare 2} and the H\"older inequality, we obtain
\begin{equation}\label{sp ie}
\inf_{c\in\re}\Big(\,\vint{B(x,2^{-k})}|u(y)-c|^{\eta}\,d\mu(y)\Big)^{1/\eta}
\le C2^{-k\eps'}\sum_{j\ge k-2}2^{-j(s-\eps')}
\Big(\,\vint{B(x,2^{-k+1})}g_j^{t}\,d\mu\Big)^{1/t}.
\end{equation}
Let $c\in\re$. By \eqref{median ie}, we have
\[
\begin{split}
|m_u(B)-m_u(B(x,2^{-k}))|
&\le |m_u(B)-c|+|c-m_u(B(x,2^{-k}))|\\
&\le \Big(2\vint{B}|u-c|^{\eta}\,d\mu\Big)^{1/\eta}+\Big(2\vint{B(x,2^{-k})}|u-c|^{\eta}\,d\mu\Big)^{1/\eta}\\
&\le C\Big(\,\vint{B(x,2^{-k})}|u-c|^{\eta}\,d\mu\Big)^{1/\eta}.
\end{split}
\]
The claim follows by taking the infimum over $c\in\re$ and applying \eqref{sp ie}.
\end{proof}

\begin{remark}\label{poincare F}
If a set $S$ satisfies measure density condition \eqref{measure density}, then the induced space $(S,d,\mu|_{S})$ satisfies 
doubling condition \eqref{doubling measure} locally, that is, for small radii,
and we can replace small balls $B$, which are centered in $S$, with $B\cap S$ in inequality \eqref{estimate for medians}.
\end{remark}

\subsection{Leibniz type rules and norm estimates for Lipschitz functions}\label{section: lip}
We finish this section by proving a Leibniz type rule for fractional $s$-gradients and some norm estimates for Lipschitz 
functions.
These norm estimates are used later to show that the extension property  for Besov spaces, or for Triebel--Lizorkin spaces,
implies measure density condition \eqref{measure density}.

\begin{lemma}\label{LemWithLipForTL}
Let $0<s<1$, $0<p<\infty$ and $0<q\le\infty$, and let $S\subset X$ be a measurable set.
Let $u\colon X\to\re$ be a measurable function with $(g_k)\in \D^s(u)$ and
let $\ph$ be a bounded $L$-Lipschitz function supported in $S$.
Then sequences $(h_k)_{k\in\z}$ and $(\rho_k)_{k\in\z}$, where
\[
\rho_k=\big(g_k\Vert \ph\Vert_{\infty}+2^{k(s-1)}L|u|\big)\ch{\operatorname{supp}\ph}
 \quad\text{and}\quad
h_k=\big(g_k+2^{s k+2}|u|\big)\Vert \ph\Vert_{\infty}\ch{\operatorname{supp}\ph}
\]
are fractional $s$-gradients of $u\ph$.
Moreover, if $u\in M^s_{p,q}(S)$, then $u\ph\in M^s_{p,q}(X)$ and $\|u\ph\|_{M^s_{p,q}(X)}\le C\|u\|_{M^s_{p,q}(S)}$.
\end{lemma}

\begin{proof}
For the first claim, let $x,y\in X$, and let $k\in\z$ such that $2^{-k-1}\le d(x,y)<2^{-k}$.
By the triangle inequality, we have
\[
|u(x)\ph(x)-u(y)\ph(y)|\le|u(x)||\ph(x)-\ph(y)|+|\ph(y)||u(x)-u(y)|.
\]
We consider four cases depending whether $x$ or $y$ belongs to $\operatorname{supp}\ph$ or not.
If $x,y\in \operatorname{supp}\ph$, then
\begin{align*}
 |u(x)\ph(x)-u(y)\ph(y)|&\le|u(x)|Ld(x,y)+\|\ph\|_{\infty}d(x,y)^s (g_k(x)+g_k(y))\\
&\le d(x,y)^s\big(2^{k(s-1)}L|u(x)|+\|\ph\|_{\infty}(g_k(x)+g_k(y))\big)\\
&\le d(x,y)^s (\rho_k(x)+\rho_k(y)),
\end{align*}
and, on the other hand,
\begin{align*}
|u(x)\ph(x)-u(y)\ph(y)|
&\le 2\|\ph\|_{\infty}|u(x)|+\|\ph\|_{L^\infty}d(x,y)^s (g_k(x)+g_k(y))\\
&\le d(x,y)^s\Vert \ph\Vert_{\infty}\big(2\cdot 2^{s (k+1)}|u(x)|+g_k(x)+g_k(y)\big)\\
&\le d(x,y)^s (h_k(x)+h_k(y)).
\end{align*}
Hence, in this case, $(\rho_k)_{k\in\z}$ and $(h_k)_{k\in\z}$ satisfy the required inequality.
The remaining two cases are considered in the same, even simpler, way.
This shows that $(\rho_k)_{k\in\z}$ and $(h_k)_{k\in\z}$ are fractional $s$-gradients of $u\ph$.

To prove the second claim, suppose that $\|\vec{g}\|_{L^p(S,\,l^q)}\le 2\inf \|\vec{r}\|_{L^p(S,\,l^q)}$,
where the infimum is taken over fractional $s$-gradients of $u$ in $S$.
By the first part of the proof, the sequence $(g'_k)_{k\in\z}$, 
\[
g'_k=
\begin{cases}
h_k,\quad &\text{if }k<k_L,\\
\rho_k,\quad &\text{if }k\ge k_L,
\end{cases}
\]
where $k_L$ is an integer such that $2^{k_L-1}<L\le 2^{k_L}$, 
is a fractional $s$-gradient of $u\ph$.

Concerning the norm, if $0<q<\infty$, we have
\[
\begin{split}
\Big(\sum_{k\in\z}|g'_k|^q\Big)^{1/q}
&\le C \Big(\|\ph\|_{\infty}\Big(\sum_{k=-\infty}^{k_L-1}(g_k+2^{s k+2}|u|)^q\Big)^{1/q}\\
   &\quad\quad+\Big(\sum_{k=k_L}^{\infty}(g_k\Vert \ph\Vert_{\infty}+2^{k(s-1)}L|u|)^q\Big)^{1/q}
   \Big)\\
&\le C\bigg(\|\ph\|_{\infty}\Big(\sum_{k\in\z}|g_k|^q\Big)^{1/q}+|u|\Big(\sum_{k=-\infty}^{k_L-1}2^{(s k+2)q}\Big)^{1/q}\\
&\quad\quad +L|u|\Big(\sum_{k=k_L}^{\infty}2^{kq(s-1)}\Big)^{1/q}\bigg),
\end{split}
\]
and hence
\[
\begin{split}
\|\vec{g'}\|_{L^p(X,\,l^q)}
&\le C\big(\Vert\ph\Vert_{\infty}\|\vec{g}\|_{L^p(S,\,l^q)}+\|u\|_{L^p(S)}2^{s k_L}+
L\|u\|_{L^p(S)}2^{k_L(s-1)}\big)\\
&\le C\big(\Vert\ph\Vert_{\infty}\|\vec{g}\|_{L^p(S,\,l^q)}+L^s\|u\|_{L^p(S)}\big).
\end{split}
\]
The claim follows by the selection of $(g_k)_{k\in\z}$.
The case $q=\infty$ follows using similar arguments.
\end{proof}

\begin{remark}\label{WithLipForB}
An analogue of Lemma \ref{LemWithLipForTL} holds also for functions from Haj{\l}asz--Besov spaces
$N^{s}_{p,q}(S)$. To prove this, it remains to show that
$\|\vec{g'}\|_{l^q(X,L^p)}<\infty$, with the corresponding bound for the norm, whenever $\vec{g}\in \D^s(u)$ is such that
$\|\vec{g}\|_{l^q(S,L^p)}<\infty$. Indeed, when $0<q<\infty$, we have
\[
\begin{split}
\|\vec{g'}\|_{l^q(X,L^p)}
&=\bigg(\sum_{k\in\z}\Vert g'_k\Vert_{L^p(X)}^q\bigg)^{1/q}\\
&=\Vert\ph\Vert_{L^\infty}\bigg(\sum_{k=-\infty}^{k_L-1}\big\|g_k
   +2^{sk+2}|u|\big\|_{L^p(S)}^q\bigg)^{1/q}\\
&\quad+\bigg(\sum_{k=k_L}^{\infty}\big\|g_k\Vert \ph\Vert_{\infty}
  +2^{k(s-1)}L|u|\big\|_{L^p(S)}^q\bigg)^{1/q}\\
&\le C\Vert\ph\Vert_{\infty}\bigg(\sum_{k\in\z}\|g_k\|_{L^p(S)}^q\bigg)^{1/q}
  +\|u\|_{L^p(S)}\bigg(\sum_{k=-\infty}^{k_L-1}2^{(sk+2)q}\bigg)^{1/q}\\
&\quad  +L\|u\|_{L^p(S)}\bigg(\sum_{k=k_L}^{\infty}2^{kq(s-1)}\bigg)^{1/q}\\
&\le C \Vert\ph\Vert_{\infty}\|\vec{g}\|_{l^q(S,L^p)}+L^s\|u\|_{L^p(S)},
\end{split}
\]
which implies the claim. The case $q=\infty$ follows similarly.
\end{remark}

By selecting $u\equiv 1$ and $g_k\equiv 0$ for all $k\in\z$ in (the proof of) Lemma \ref{WithLipForB},
we obtain norm estimates for Lipschitz functions supported in bounded sets.

\begin{corollary}\label{lemma:lip frac gradient}
Let $0<s<1$, $0<p<\infty$ and $0<q\le\infty$.
Let $\Omega\subset X$ be a measurable set and
let $\ph\colon \Omega\to\re$ be an $L$-Lipschitz function supported in a bounded set $F\subset \Omega$.
Then $\ph\in M^{s}_{p,q}(\Omega)$ and
\begin{equation}\label{Msps norm u}
\|\ph\|_{M^{s}_{p,q}(\Omega)}\le C(1+\|\ph\|_{\infty})(1+L^{s})\mu(F)^{1/p},
\end{equation}
where the constant $C>0$ depends only on $s$ and $q$.
The claim holds also with $M^{s}_{p,q}(\Omega)$ replaced by $N^{s}_{p,q}(\Omega)$.
 \end{corollary}

\section{Interpolation and embedding theorems for Besov spaces}\label{section: interpolation}

In this section, we prove new interpolation and embedding theorems for Besov spaces.
Recall some essential definitions and properties of the real interpolation theory;
see, for example, the classical references \cite{BS}, \cite{BL} for the details.

Let $A_0$ and $A_1$ be (quasi-semi)normed spaces continuously embedded into a topological vector space $\mathcal{A}$.
For every $f\in A_0+A_1$ and $t>0$, the {\em $K$-functional} is
\[
K(f,t;A_0,A_1)=\inf\big\{\|f_0\|_{A_0}+t\|f_1\|_{A_1}: f=f_0+f_1\big\}.
\]
Let $0<s<1$ and $0<q\le \infty$.
The {\em interpolation space} $(A_0,A_1)_{s,q}$ consists of functions $f\in A_0+A_1$, for which
\[
\|f\|_{(A_0,A_1)_{s,q}}=
\begin{cases}
\Big(\int_0^\infty\big(t^{-s}K(f,t;A_0,A_1) \big)^q \frac{dt}{t}\Big)^{1/q}, &\text{ if }q<\infty\\
\sup_{t>0}t^{-s}K(f,t;A_0,A_1), &\text{ if }q=\infty,
\end{cases}
\]
is finite.

The following theorem is the main result of this section.
We will apply it later only in the case $q=\infty$, but since this interpolation result is of independent interest,
we prove it in full generality.
The case $1\le p<\infty$, $1\le q\le\infty $ was earlier obtained  in \cite[Cor. 4.3]{GKS}
using a version of the Korevaar--Schoen definition for the Sobolev spaces in the metric setting.

\begin{theorem}\label{thm: interpolation}
Let $X$ be a metric space with a doubling measure $\mu$.
Let $0<p<\infty$, $0<q\le \infty$ and $0<s<1$. Then
\begin{equation}\label{eq: interpolation homogeneous}
\dot N^s_{p,q}(X)=(L^p(X),\dot M^{1,p}(X))_{s,q}
\end{equation}
and
\begin{equation}\label{eq: interpolation}
N^s_{p,q}(X)=\big(L^p(X),M^{1,p}(X)\big)_{s,q}
\end{equation}
with equivalent norms.
\end{theorem}

\begin{proof}
To prove \eqref{eq: interpolation homogeneous}, we first show that there exists a constant $C>0$ such that for all $t>0$,
\begin{equation}\label{KfunctionalEstimate}
C^{-1}E_p(f,t)\le K(f,t;L^p(X),\dot M^{1,p}(X))
\le C\bigg(\sum_{k=0}^{\infty}2^{-k\tilde p}E_p^{\tilde p}(f,2^{k}t)\bigg)^{1/\tilde p},
\end{equation}
where $\tilde p=\min\{p,1\}$ and $E_p(f,t)$ is as in \eqref{eput}.

We begin with the first inequality in \eqref{KfunctionalEstimate}.
Let $f=g+h$, where $g\in L^p(X)$ and $h\in \dot M^{1,p}(X)$, and let $t>0$. Then
\[
E_p(f,t)\le C(E_p(g,t)+E_p(h,t)),
\]
where, by the Fubini theorem,
\begin{equation}\label{B int}
\begin{aligned}
E^p_p(g,t)
&\le 2^p\int_X |g(x)|^p\,d\mu(x)+2^p\int_X\vint{B(x,t)}|g(y)|^p\,d\mu(y)d\mu(x)\\
&= 2^p\int_X |g(x)|^p\,d\mu(x)+2^p\int_X|g(y)|^p\int_{B(y,t)}\frac{1}{\mu(B(x,t))}\,d\mu(x)\,d\mu(y)\\
&\le C\|g\|_{L^p(X)}^p.
\end{aligned}
\end{equation}
The last estimate follows using the doubling property of $\mu$ and the fact that $B(y,t)\subset B(x,2t)$ for each $x\in B(y,t)$.

By the definition of the $1$-gradient and by the similar argument as in
\eqref{B int}, for every $\rho\in \mathcal D(h)\cap L^p(X)$, we have,
\[
\begin{split}
E^p_p(h,t)&=\int_X\vint{B(x,t)}|h(x)-h(y)|^p\,d\mu(y)d\mu(x)\\
&\le \int_X\vint{B(x,t)}(d(x,y))^p(\rho(x)+\rho(y))^p\,d\mu(y)d\mu(x)\\
&\le Ct^p\Big(\int_X \rho(x)^p\,d\mu(x)+ \int_X\vint{B(x,t)}\rho(y)^p\,d\mu(y)d\mu(x)\Big)\\
&\le Ct^p\|\rho\|_{L^p(X)}^p\,.
\end{split}
\]
By taking the infimum over all representations of $f$ in $L^p(X)+\dot M^{1,p}(X)$, we have that
$E_p(f,t)\le CK(f,t;L^p(X),\dot M^{1,p}(X))$.
\medskip

To prove the second inequality in \eqref{KfunctionalEstimate}, let $f\in L^p_{\text{loc}}(X)$ and let $t>0$.
By a standard covering argument, there is a covering of $X$ by balls $B_i=B(x_i, t/6)$, $i\in\n$,
such that $\sum_i\ch{2B_i}\le N$ with the overlap constant $N>0$ depending only on the doubling constant of $\mu$.

Let $\{\ph_i\}_{i\in\n}$ be a collection of $Ct^{-1}$-Lipschitz functions $\ph_i\colon X\to[0,1]$
such that $\operatorname{supp}\ph_i\subset 2B_i$ and $\sum_i\ph_i(x)=1$ for all $x\in X$, (this is a so called partition of unity 
subordinate to the covering $\{B_i\}_{i\in\n}$, see also the beginning of Section \ref{section: extensions}).

Let $h\colon X\to \re$ be a function defined using median values \eqref{median} of $f$,
\[
h(x)=\sum_{i\in\n}m_f(B_i)\ph_i(x),\quad\text{for all }x\in X,
\]
and let $g=f-h$.

Let $x\in X$ and let $I_x=\{i:\,x\in 2B_i\}$. By the properties of the partition of unity,
\[
g(x)=\sum_{i\in\n}\big(f(x)-m_f(B_i)\big)\ph_i(x)
=\sum_{i\in I_x}(f(x)-m_f(B_i))\ph_i(x).
\]
Since the number of elements in $I_x$ is bounded by the overlap constant $N$ independent of $x$ and $t$,
and, for every $i\in I_x$, $B_i\subset B(x,t)\subset 8B_i$, using \eqref{median ie} and the doubling property of $\mu$, we 
obtain
\begin{equation}\label{estimate for g}
|g(x)|
\le 2\sum_{i\in I_x}\Big(\,\vint{B_i}|f(x)-f(z)|^p\,d\mu(z)\Big)^{1/p}
\le C\Big(\,\vint{B(x,t)}|f(x)-f(z)|^p\,d\mu(z)\Big)^{1/p},
\end{equation}
which implies that
\begin{equation}\label{gEp}
\| g\|_{L^p(X)}\le CE_p(f,t).
\end{equation}

Next we estimate $h$ in the $ \dot M^{1,p}$-norm. Let $t>0$ and let $x,y\in X$. We consider two cases.

\noindent{\sc Case 1}:
If $d(x,y)\le t$, then $B_i\subset B(x,2t)\subset 20B_i$, for every $i\in I_x\cup I_y$.
Using the properties of the functions $\ph_i$, we have
\[
\begin{split}
h(x)-h(y)&=\sum_{i\in\n}(m_f(B_i)-f(x))(\ph_i(x)-\ph_i(y))\\
&=\sum_{i\in I_x\bigcup I_y}(m_f(B_i)-f(x))(\ph_i(x)-\ph_i(y)),
\end{split}
\]
which together with the $Ct^{-1}$-Lipschitz continuity of the functions $\ph_i$ and \eqref{median ie} implies that
 \[
\begin{split}
|h(x)-h(y)|
&\le C\,\frac{d(x,y)}{t}\sum_{i\in I_x\cup I_y}\Big(\,\vint{B_i}|f(x)-f(z)|^p\,d\mu(z)\Big)^{1/p}\\
&\le C\,\frac{d(x,y)}{t}\Big(\,\vint{B(x,2t)}|f(x)-f(z)|^p\,d\mu(z)\Big)^{1/p}.
\end{split}
\]

\noindent{\sc Case 2}:
Let $d(x,y)>t$. Since
\[
|h(x)-h(y)|\le |f(x)-f(y)|+|g(x)|+|g(y)|,
\]
it suffices to estimate the terms on the right side.
The assumption $d(x,y)>t$ and \eqref{estimate for g} imply that
\[
|g(x)|\le C\,\frac{d(x,y)}{t}\Big(\,\vint{B(x,t)}|f(z)-f(x)|^p\,d\mu(z)\Big)^{1/p},
\]
and a corresponding upper bound holds for $|g(y)|$.

Using \eqref{median ie} and the doubling property of $\mu$ and writing $R=d(x,y)$, we obtain
\begin{align*}
 |f(x)-f(y)|
 &\le |f(x)-m_f(B(x,R))|+|f(y)-m_f(B(x,R))|\\
&\le 2\Big(\,\vint{B(x,R)}|f(z)-f(x)|^p\,d\mu(z)\Big)^{1/p}\\
&\quad
+C\Big(\,\vint{B(y,2R)}|f(z)-f(y)|^p\,d\mu(z)\Big)^{1/p},
\end{align*}
and, hence,
\[
|f(x)-f(y)|\le Cd(x,y)(f^\sharp_t(x)+f^\sharp_t(y)),
\]
where
\[
f^\sharp_t(x)
=\sup_{r\ge t}\frac{1}{r}\Big(\,\vint{B(x,r)}|f(z)-f(x)|^p\,d\mu(z)\Big)^{1/p}\,.
\]
Collecting the estimates, we obtain, in both cases, that
\[
|h(x)-h(y)|\le Cd(x,y)(f^\sharp_t(x)+f^\sharp_t(y)),
\]
which shows that $f^\sharp_t\in\mathcal D(h)$. Hence it suffices to estimate $\|f^\sharp_t\|_{L^p(X)}$.

Using the definition of $f^\sharp_t$ and the doubling property of $\mu$, we have
\begin{align*}
f^\sharp_t(x)
&\le \sum_{k=0}^{\infty}\sup_{2^{k-1}t<r\le 2^{k}t}\frac{1}{r}\Big(\,\vint{B(x,r)}|f(z)-f(x)|^p\,d\mu(z)\Big)^{1/p}\\
&\le \frac{C}{t}\sum_{k=0}^{\infty}2^{-k}\Big(\,\vint{B(x,2^{k} t)}|f(z)-f(x)|^p\,d\mu(z)\Big)^{1/p}\,.
\end{align*}
If $0<p\le 1$, we use inequality \eqref{a sum} and obtain
\begin{equation}\label{fEp1}
\begin{aligned}
\|f^\sharp_t\|^p_{L^p(X)}
&\le \frac{C}{t^p}\sum_{k=0}^{\infty}2^{-kp}\int_X\vint{B(x,2^{k} t)}|f(z)-f(x)|^p\,d\mu(z)\,d\mu(x)\\
&=\frac{C}{t^p}\sum_{k=0}^{\infty}2^{-kp}E_p^{p}(f,2^{k}t)\,.
\end{aligned}
\end{equation}

When $p>1$, we have, by the Minkowski inequality, that
\begin{equation}\label{fEp2}
\begin{aligned}
\|f^\sharp_t\|_{L^p(X)}\le \frac{C}{t}\sum_{k=0}^\infty 2^{-k}E_p(f,2^kt).
\end{aligned}
\end{equation}
Thus, the required inequality
\[
K(f,t;L^p(X), \dot M^{1,p}(X))
\le C\Big(\sum_{k=0}^{\infty}2^{-k\tilde p}E_p^{\tilde p}(f,2^{k}t)\Big)^{1/\tilde p}
\]
follows using \eqref{gEp}-\eqref{fEp2} and the definition of the $K$-functional.
\smallskip

\noindent{\bf Interpolation result \eqref{eq: interpolation homogeneous} for $\dot N^{s}_{p,q}(X)$:}
The equivalence of Besov norms \eqref{equivalence of Besov norms}, the definition of the norm $\|f\|_{\dot\cB^s_{p,q}(X)}$ 
and
the first inequality in \eqref{KfunctionalEstimate} imply that
\[
\|f\|_{\dot N^s_{p,q}(X)}\le C\|f\|_{(L^p(X),\dot M^{1,p}(X))_{s,q}}.
\]
To obtain the opposite estimate, we use the second inequality of \eqref{KfunctionalEstimate}.

If $q<\infty$, then
\[
\begin{split}
\|f\|_{(L^p(X),\dot M^{1,p}(X))_{s,q}}&=
\bigg(\int_0^\infty\big(t^{- s}K(f,t;L^p(X),\dot M^{1,p}(X))\big)^q\frac{dt}{t}\bigg)^{1/q}\\
&\le C\bigg(\int_0^\infty \bigg(\sum_{k=0}^{\infty}
     t^{- s \tilde p}2^{-k\tilde p}E_p^{\tilde p}(f,2^{k}t)\bigg)^{q/\tilde p}\frac{dt}{t}\bigg)^{1/q},
\end{split}
\]
where the last estimate is denoted by $A$.
If $q\le \tilde p$, then using \eqref{a sum} and change of variables, we obtain
\[
\begin{split}
A
&\le \bigg(\sum_{k=0}^{\infty}2^{-kq}\int_0^\infty\Big(t^{- s}E_p(f,2^{k}t)\Big)^{q}\frac{dt}{t}\bigg)^{1/q}\\
&= \bigg(\sum_{k=0}^{\infty}2^{-kq}\int_0^\infty\Big((2^{-k}\tau)^{- s}E_p(f,\tau)\Big)^{q}\frac{d\tau}{\tau}\bigg)^{1/q}\\
&= \bigg(\sum_{k=0}^{\infty}2^{kq( s-1)}\bigg)^{1/q}\bigg( \int_0^\infty\Big(\tau^{- s}E_p(f,\tau)\Big)^{q}\frac{d\tau}{\tau}
\bigg)^{1/q}.
\end{split}
\]
If $q\geq \tilde p$, then, using the Minkowski inequality and changing variables, we have
\[
\begin{split}
A
&\le \bigg(\sum_{k=0}^{\infty}\bigg(\int_0^\infty\Big(t^{- s \tilde p}2^{-k\tilde p}E_p^{\tilde p}(f,2^{k}t)\Big)^{q/\tilde p}\frac{dt}{t}
\bigg)^{\tilde p/q}\bigg)^{1/\tilde p}\\
&=\bigg(\sum_{k=0}^{\infty}\bigg(\int_0^\infty\Big((2^{-k}\tau)^{- s\tilde p}2^{-k\tilde p}E_p^{\tilde p}(f,\tau)\Big)^{q/\tilde p}\frac{d
\tau}{\tau}\bigg)^{\tilde p/q}\bigg)^{1/\tilde p}\\
&=\bigg(\sum_{k=0}^{\infty}2^{k\tilde p( s-1)}\bigg)^{1/\tilde p}\bigg(\int_0^\infty\Big(\tau^{- s}E_p(f,\tau)\Big)^{q}\frac{d\tau}
{\tau}\bigg)^{1/q}.
\end{split}
\]
In the case $q=\infty$, we have, using \eqref{KfunctionalEstimate},
\[
\begin{split}
\|f\|_{(L^p(X),\dot M^{1,p}(X))_{s,\infty}}
&=\sup_{t>0}t^{-s}K(f,t;L^p(X),\dot M^{1,p}(X))\\
&\le C\sup_{t>0}t^{-s}\bigg(\sum_{k=0}^{\infty}2^{-k\tilde p}E_p^{\tilde p}(f,2^{k}t)\bigg)^{1/\tilde p}\\
&= C\sup_{t>0}\bigg(\sum_{k=0}^{\infty}2^{k\tilde p(s-1)}\Big((2^kt)^{-s}E_p(f,2^{k}t)\Big)^{\tilde p} \bigg)^{1/\tilde p}\\
&\le C\Big(\sum_{k=0}^{\infty}2^{k\tilde p( s-1)}\Big)^{1/\tilde p}\sup_{t>0}t^{-s}E_p(f,t).
\end{split}
\]

By the comparability of the norms in Besov spaces given by different definitions, \eqref{equivalence of Besov norms}, we have
\[
\|f\|_{(L^p(X),\dot M^{1,p}(X))_{s,q}}\le C\|f\|_{\dot N^s_{p,q}(X)}.
\]
\smallskip

\noindent{\bf Interpolation result \eqref{eq: interpolation} for $N^{s}_{p,q}(X)$:}
We start by showing that
\begin{equation}\label{2K}
K(f,t,L^p(X),M^{1,p}(X))\approx K(f,t;L^p(X),\dot M^{1,p}(X))+ \min\{1,t\}\|f\|_{L^p(X)}
\end{equation}
for each $f\in L^p(X)+M^{1,p}(X)$ and every $t>0$.

Let $f$ be such a function and let $t>0$.
The definition of the $K$-functional and the spaces $\dot M^{1,p}(X)$ and $M^{1,p}(X)$ imply that
\[
K(f,t;L^p(X),\dot M^{1,p}(X))\le K(f,t,L^p(X),M^{1,p}(X)).
\]
For every $g\in L^p(X)$ and $h\in M^{1,p}(X)$ with $f=g+h$, we have
\[
\begin{split}
\min\{1,t\}\|f\|_{L^p(X)}
&\le C\big(\min\{1,t\}\|g\|_{L^p(X)}+\min\{1,t\}\|h\|_{L^p(X)}\big)\\
&\le C\big(\|g\|_{L^p(X)}+t\|h\|_{M^{1,p}(X)}\big),
\end{split}
\]
which implies that
\[
\min\{1,t\}\|f\|_{L^p(X)}\le CK(f,t,L^p(X),M^{1,p}(X)).
\]
This implies one direction of inequality \eqref{2K}.

For the other direction, assume first that $t>1$. Then the claim follows from the fact that
\[
K(f,t,L^p(X),M^{1,p}(X))\le \|f\|_{L^p(X)}.
\]
If $0<t<1$, let $g\in L^p(X)$ and $h\in \dot M^{1,p}(X)$ be such that $f=g+h$.
Then $h\in L^p(X)$ with $\|h\|_{L^p(X)}\le C(\|f\|_{L^p(X)}+\|g\|_{L^p(X)})$ and
\begin{align*}
K(f,t,L^p(X),M^{1,p}(X))
&\le  \|g\|_{L^p(X)}+t\|h\|_{M^{1,p}(X)}\\
&= \|g\|_{L^p(X)}+t\|h\|_{L^p(X)}+t\|h\|_{\dot M^{1,p}(X)}\\
&\le C\big(\|g\|_{L^p(X)}+t\|h\|_{\dot M^{1,p}(X)}+t\|f\|_{L^p(X)}\big),
\end{align*}
from which the claim follows by taking the infimum over such decompositions $f=g+h$.

Since
\[
\int_0^\infty\big(t^{- s}\min\{1,t\}\big)^q\frac{dt}{t}<\infty \ \ \text{ and } \ \ \sup_{t>0}t^{-s}\min\{1,t\}<\infty,
\]
\eqref{2K} implies that
\[
\|f\|_{(L^p(X),M^{1,p}(X))_{s,q}}\approx \|f\|_{(L^p(X),\dot M^{1,p}(X))_{s,q}}+\|f\|_{L^p(X)},
\]
and, hence, \eqref{eq: interpolation} follows from \eqref{eq: interpolation homogeneous}.
\end{proof}

\begin{remark}\label{extension interpolation}
Since each linear operator which is bounded in $L^p$ and in $M^{1,p}$, is bounded in the interpolation space,
the extension results for Besov space with $p\ge 1$ follow from Theorem \ref{thm: interpolation} and the extension results in 
\cite{HKT_Revista}.
\end{remark}

Theorem \ref{thm: interpolation} and the reiteration theorem \cite[Thm 3.1]{Ho} imply the following interpolation theorem for 
Haj\l asz--Besov spaces. In the Euclidean setting, this result was proved in \cite{DP} using different methods.
For related interpolation results in the metric setting, see \cite{Y2}, \cite{HMY} and \cite{GKS}.

\begin{theorem}
Let $X$ be a metric space with a doubling measure $\mu$.
Let $0<p<\infty$, $0<q,q_0,q_1\le\infty$, $0<s_0,s_1,\lambda<1,$ and $s=(1-\lambda)s_0+\lambda s_1$.
Then
\[
\big(\dot N^{s_0}_{p,q_0}(X),\dot N^{s_1}_{p,q_1}(X) \big)_{\lambda,q}=\dot N^s_{p,q}(X)
\]
and
\[
\big( N^{s_0}_{p,q_0}(X),N^{s_1}_{p,q_1}(X) \big)_{\lambda,q}=N^s_{p,q}(X)
\]
with equivalent norms.
\end{theorem}

\subsection{An Embedding theorem for the Haj\l  asz--Besov spaces}
Our interpolation theorem implies a Sobolev type embedding result for the Haj\l asz--Besov spaces.
The embedding is into the Lorentz spaces.

Recall that
the Lorentz space $L^{p,q}(X)$, $0<p<\infty$, $0<q\le\infty$, consists of measurable functions
$u\colon X\to[-\infty,\infty]$, for which the (quasi)norm
\[
\| u\|_{L^{p,q}(X)}=p^{1/q}\bigg(\int_0^\infty t^{q}\mu\big(\{x\in X:\,|u(x)|\geq t\}\big)^{q/p}\,\frac{d t}{t}\bigg)^{1/q}\,,
\]
when $q<\infty$, and
\[
\| u\|_{L^{p,\infty}(X)}=\sup_{t>0} t \mu\big(\{x\in X:\,|u(x)|>t\}\big)^{1/p},
\]
when $q=\infty$, is finite. Using the Cavalieri principle, it is easy to see that $L^{p,p}(X)=L^p(X)$.
Moreover, $L^{p,\infty}(X)$ equals weak $L^p(X)$-space and $L^{p,q}(X)\subset L^{p,r}(X)$ when $r>q$.

In the Euclidean setting, embedding $\cB^s_{p,q}(\rn)\hookrightarrow L^{p^*(s),q}(\rn)$ was obtained in \cite[Thm 1.15]{HS} 
using an
atomic decomposition of $\cB^s_{p,q}(\rn)$.
In the metric case, the embedding of Besov spaces $\cB^s_{p,q}(X)$,
$p>1$, $q\ge 1$, to Lorentz spaces was proved in \cite{GKS} under the assumption that $X$ supports a $(1,p)$-Poincar\'e 
inequality.
The idea of our proof comes from \cite[Thm 5.1]{GKS}.
For the readers' convenience, we give the proof with all details.

\begin{theorem}\label{thm: Lorentz-Sobolev}
Let $X$ be a $Q$-regular metric space, $Q\ge 1$.
Let $0<s<1$, $0<p<Q/s$ and $0<q\le\infty$.
There is a constant $C>0$ such that
\begin{equation}\label{WeakTypeBesov-Poincare}
\inf_{c\in\re}\Vert u-c\Vert_{L^{p^*(s),q}(X)}
\le C\Vert u\Vert_{\dot N^{s}_{p,q}(X)}\,,
\end{equation}
where $p^*(s)=Qp/(Q-sp)$.
\end{theorem}
\begin{proof}
By Lemma \ref{lemma:Sobolev-Poincare 1}, for every ball $B(x,r)\subset X$ and for every $u\in \dot M^{1,p}(X)$, we have
\[
\inf_{c\in\re}\Big(\,\vint{B(x,r)}|u(y)-c|^{p^*}\,d\mu(y)\Big)^{1/p^*}\\
\le Cr\Big(\,\vint{B(x,2r)}g^{p}\,d\mu\Big)^{1/p},
\]
where $p^*=pQ/(Q-p)$, whenever $g\in L^p(X)$ is a $1$-gradient of $u$.
Since the measure $\mu$ is $Q$-regular, we obtain
\[
\inf_{c\in\re}\Big(\int_{B(x,r)}|u(y)-c|^{p^*}\,d\mu(y)\Big)^{1/p^*}\\
\le C_0\Big(\int_{B(x,2r)}g^{p}\,d\mu\Big)^{1/p},
\]
where the constant $C_0>0$ is independent of $r$.
It follows that, for every $k\ge 1$, there is $c_k\in\re$ such that
\[
\|u-c_k\|_{L^{p*}(B(x,k))}\le 2C_0\|u\|_{\dot M^{1,p}(X)}.
\]
Since
\[
\|u-c_k\|_{L^{p*}(B(x,1))}
\le \|u-c_k\|_{L^{p*}(B(x,k)}
\le 2C_0\|u\|_{\dot M^{1,p}(X)},
\]
and $X$ is $Q$-regular, we have that
\begin{align*}
|c_k|
&\le c_Q^{1/p^*}\big(\|u-c_k\|_{L^{p*}(B(x,1))}+\|u\|_{L^{p*}(B(x,1))}\big)\\
&\le c_Q^{1/p^*}\big(\|u-c_k\|_{L^{p*}(B(x,1))}+\|u-c_1\|_{L^{p*}(B(x,1))}+c_1c_Q^{1/p^*}\big)\\
&\le C \|u\|_{\dot M^{1,p}(X)} +C,
\end{align*}
where the constant $C>0$ does not depend on $k$.
As a bounded sequence in $\re$, $(c_k)$ has a subsequence $(c_{k_j})$ that converges to some $c\in\re$.

Now, for a fixed $m$, and for each $k_j\ge m$, we have
\[
\begin{split}
\|u-c\|_{L^{p*}(B(x,m))}
&\le C\big( \|u-c_{k_j}\|_{L^{p*}(B(x,m))} + \|c_{k_j}-c\|_{L^{p*}(B(x,m))}\big)\\
&\le C\big( \|u-c_{k_j}\|_{L^{p*}(B(x,k_j))} + \|c_{k_j}-c\|_{L^{p*}(B(x,m))}\big)\\
&\le C\big( \|u\|_{\dot M^{1,p}(X)}+ \mu(B(x,m))^{1/p^*}|c_{k_j}-c|\big).
\end{split}
\]
By letting first $j\to\infty$ and then $m\to\infty$, we conclude that
\[
\|u-c\|_{L^{p^*}(X)}
\le C \| u\|_{\dot M^{1,p}(X)}\,.
\]
Since $\frac{1-s}{p}+\frac{s}{p^*}=\frac{1}{p^*(s)}$, an interpolation theorem from \cite[Thm 4.3]{Ho}
together with the fact that $L^{r,r}(X)=L^{r}(X)$ for each $r$, states that
\[
L^{p^*(s),q}(X)=(L^{p}(X), L^{p^*}(X))_{s,q}.
\]
Thus, using Theorem \ref{thm: interpolation}, we obtain
\[
\begin{split}
\inf_{c\in\re}\| u-c\|_{L^{p^*(s),q}(X)}
&\le C\inf_{c\in\re}\| u-c\|_{(L^{p}(X), L^{p^*}(X))_{s,q}}\\
&\le C \| u\|_{(L^{p}(X),\dot M^{1,p}(X))_{s,q}}
\approx C\| u\|_{\dot N^{s}_{p,q}(X)}.
\end{split}
\]
\end{proof}

\section{The proof of Theorem \ref{m extension}}\label{section: extensions}
In order to prove Theorem \ref{m extension}, we use a modification of the Whitney extension method,
which has been standard in the study of extension problems  starting from work \cite{J}.
We start by recalling basic properties of the Whitney covering and the corresponding partition of unity;
see, for example, \cite{HKT_Revista}. 
We also refer to \cite[Thm III.1.3]{CW} and \cite[Lemma 2.9]{MS} for the proofs of these properties.

Let $U\subsetneq X$ be an open set and, for each $x\in U$, let
\[
r(x)=\operatorname{dist}(x,X\setminus U)/10.
\]
There exists a countable family
${\cB}=\{ B_i\}_{i\in I}$  of balls $B_i=B(x_i,r_i)$, where $r_i=r(x_i)$, such that
${\cB}$ is a covering of $U$ and the balls $1/5B_i$ are disjoint.
The next lemma easily follows from the definition of the Whitney covering $\cB$ and the doubling property
of the measure $\mu$.

\begin{lemma}\label{whitney}
Let $\cB$ be a Whitney covering of an open set $U$.
There is $M\in\n$ such that for all $i\in\n$,
\begin{enumerate}
\item
\label{whitney inside U}
$5B_i\subset U$,
\item
\label{whitney distance to U}
if $x\in 5B_i$, then $5r_i<\operatorname{dist}(x,X\setminus U)< 15r_i$,
\item
\label{whitney x*i}
there is $x^*_i\in X\setminus U$ such that $d(x_i,x^*_i)<15r_i$,
\item
\label{whitney overlap}
$\sum_{i\in I}\ch{5B_i}(x)\le M$ for all $x\in U$.
\end{enumerate}
\end{lemma}

Let  $\{\ph_{i}\}_{i\in I}$ be a Lipschitz partition of unity subordinated to the covering
${\cB}$ with the following properties:

\begin{enumerate}
\item[(i)]
$\operatorname{supp}\ph_i\subset 2B_i$,
\item[(ii)]
$\ph_i(x)\ge M^{-1}$ for all $x\in B_i$,
\item[(iii)]
\label{phi lip}
there is a constant $K>0$ such that each $\ph_i$ is $Kr_i^{-1}$-Lipschitz,
\item[(iv)]
$\sum_{i\in I}\ph_i(x)=\chi_{U}(x)$.
\end{enumerate}
Note that if $5B_i\cap 5B_j\neq\emptyset$, then $1/3r_{i}\le r_j\le3 r_{j}$ and
$d(x_i^*,x_j^*)\le 80 r_{i}$, where the points $x_i^*,x_j^*$ are as in
Lemma \ref{whitney} \eqref{whitney x*i}.
\medskip

As it was already mentioned, we construct an extension operator using median values of a function.
By this technique, we can prove the result for all $0<p<\infty$, but our extension operator appears to be non-linear.
If $p>Q/(Q+s)$,
a linear extension can be obtained by replacing medians $m_u(B^*_i\cap S)$
with integral averages $u_{B^*_i\cap S}$ in the definition of the local extension \eqref{Eu}. This is easy to show
employing \eqref{Poincare 2} in the proof;
we leave the details to the reader.

\subsection{The proof of Theorem \ref{m extension}}
Let $S\subsetneq X$ be a set satisfying measure density condition \eqref{measure density}.
We may assume that $S$ is closed, because
$\mu(\overline{S}\setminus S)=0$ by \cite[Lemma 2.1]{ShMetric}.

Assume first that $u\in M^s_{p,q}(S)$ and that $(g_{k})_{k\in\z}\in \D^{s}(u)$ with
\[
\|(g_{k})\|_{L^p(S,\,l^q)}<2\inf_{{(h_{k})\in \D^{s}(u)}}\|(h_{k})\|_{L^p(S,\,l^q)}.
\]
Although the functions $g_k$ are defined on $S$ only, we identify them with functions defined on $X$ by assuming that each
$g_k=0$ on $X\setminus S$.

Let ${\cB}=\{B_i\}_{i\in I}$, $B_i=B(x_i,r_i)$, be a Whitney covering of $X\setminus S$ and let
$\{\ph_i\}_{i\in I}$ be the associated Lipschitz partition of unity.
Define ${\cB}_1 = \{ B_i\}_{i\in J}$  as the collection of all balls from ${\cB}$ with radius less than $1$, and note that the 
measure density condition holds for balls in ${\cB}_1$.
For each $i\in J$, let $x_i^*$ be "the closest point of $x_i$ in $S$" as in Lemma~\ref{whitney}, let
\[
B^*_i=B(x^*_i,r_i),
\]
and for each $x\in 2B_i$, $i\in J$, let
\[
B_x=B(x,25 r(x))=B(x,\tfrac52\operatorname{dist}(x,S)).
\]
Then $B_i^*\subset B_{x}\subset 47B_i^*$ and, by the measure density condition and the doubling property of $\mu$,
\[
\mu(B_x)\le C\mu(B_i^*\cap S).
\]
\smallskip

\noindent{\bf A local extension to the neighborhood of $S$:}

We will first construct an extension of $u$ with norm estimates to set
\[
V= \{x\in X:\, \operatorname {dist}(x,S)<8\}.
\]
For each $x\in V\setminus S$, let
\[
I_x=\{i\in I:  x\in 2B_i\}.
\]
By Lemma~\ref{whitney},  the number of elements in $I_x$ is bounded by $M$.
Moreover, if $i\in I\setminus J$, then $r_i\ge 1$ and hence
$\operatorname {dist}(2B_i,S)\ge 8r_i\geq 8$. Thus $2B_i\cap V=\emptyset$ and $i\not\in I_x$.
Accordingly, $I_x\subset J$ and therefore
\[
\sum_{i\in I_x} \ph_i(x)=
\sum_{i\in I}\ph_i(x) = \sum_{i\in J} \ph_i(x) = 1
\quad \text{for $x\in V\setminus S$}.
\]
Define the local extension $\tilde{E}u$ of $u$ by
\begin{equation}\label{Eu}
\tilde{E}u(x)=
\begin{cases}
u(x),&\text{ if }x\in S,\\
\sum_{i\in J}\ph_i(x)m_u(B^*_i\cap S),&\text{ if }x\in X\setminus S.
\end{cases}
\end{equation}
We begin by showing that
\begin{equation}\label{Eu norm}
 \|\tilde Eu\|_{L^{p}(V)}\le C\|u\|_{L^{p}(S)}.
\end{equation}
If $x\in X\setminus S$,
applying \eqref{median ie} with $0<\eta<p$, we obtain
\begin{align*}
|\tilde Eu(x)|
&\le \sum_{i\in I_{x}}\ph_i(x)|m_u(B^*_i\cap S)|
\le C\sum_{i\in I_{x}}\Big(\,\vint{B_{i}^{*}\cap S}|u|^{\eta}\,d\mu\Big)^{1/\eta}\\
&\le C \Big(\,\vint{B_{x}}|u|^{\eta}\,d\mu\Big)^{1/\eta}
\le C \big(\M u^\eta(x)\big)^{1/\eta}.
\end{align*}
Note that in the estimates above,  we assumed that $u$ equals zero outside $S$, hence,
\[
\vint{B_{x}}|u|\,d\mu=\mu(B_{x})^{-1}\int_{B_{x}\cap S}|u|\,d\mu.
\]
Now norm estimate \eqref{Eu norm} follows from the definition of $\tilde Eu$ and the boundedness of the Hardy--Littlewood 
maximal operator in $L^{p/\eta}$.
\smallskip

\noindent{\bf A fractional $s$-gradient for the local extension:}

Let $0<\delta<1-s$, $0<\eps'<s$ and $0<t<\min\{p,q\}$.
We define the sequence $(\tilde g_{k})_{k\in\z}$, a candidate for the fractional $s$-gradient of $\tilde{E}u$, as follows
\begin{equation}\label{gradientOfextension}
\tilde g_{k}(x)=\sum_{j=-\infty}^{k-1}2^{(j-k)\delta}\big(\M g_j^t(x)\big)^{1/t} \ +
\sum_{j=k-6}^{\infty}2^{(k-j)(s-\eps')}\big(\M g_j^t(x)\big)^{1/t}.
\end{equation}

We will split the rest of the proof of the theorem into several steps.

\begin{lemma}\label{lemmaGradient}
There is a constant $C>0$ such that $(C\tilde g_{k})_{k\in \z}$, where functions $\tilde g_{k}$ are given by formula 
\eqref{gradientOfextension}, is a fractional $s$-gradient of $\tilde{E}u$ .
\end{lemma}

\begin{proof}
Let $k\in\z$ and let $x,y\in V$ be such that $2^{-k-1}\le d(x,y)<2^{-k}$.
We consider the following four cases:

\noindent
{\sc Case 1:} Since, clearly,  almost everywhere on $S$ the inequality $g_{k}\le\tilde{g_{k}}$ holds, for almost every $x,y\in S$,
\[
|\tilde{E}u(x) - \tilde{E}u(y)|
=| u(x)-u(y)|
\le d(x,y)^{s}(\tilde g_{k}(x)+\tilde g_{k}(y)).
\]

\noindent
{\sc Case 2:} $x\in V\setminus S$, $y\in S$.

Then $r(x)=\operatorname{dist}(x,S)/10<2^{-k}/10$ and there is $x^*\in B_x\cap S$
such that $d(x,x^*)<15 r(x)$. Let $m\in\z$ be such that $2^{-m-1}\le50r(x)<2^{-m}$ and let
\[
B_{x^*}=B(x^*,2^{-m}).
\]
Then $B_x\subset B_{x^*}$ and  $2B_{x^*}\subset 9B_x$.
Now
\begin{equation}\label{xV yF}
|\tilde{E} u(x) - \tilde{E}u(y)|
\le |\tilde{E}u(x) - m_{u}(B_{x^*}\cap S)| +|u(y) -m_{u}(B_{x^*}\cap S)|,
\end{equation}
and we begin with the first term of \eqref{xV yF}.
Using \eqref{estimate for medians} and
the fact that
$B_i^{*}\subset B_{x}\subset B_{x^{*}}$ with comparable measures, we have
\begin{align*}
 |\tilde{E}u(x) - m_{u}(B_{x^*}\cap S)|
&=\Big|\sum_{i\in I_{x}}\ph_{i}(x)(m_{u}(B_{i}^{*}\cap S)-m_{u}(B_{x^*}\cap S))\big|\\
&\le C2^{-m\eps'}\sum_{j\ge m-2}2^{-j(s-\eps')}
\Big(\,\vint{B(x^{*},2^{-m+1})}g_j^{t}\,d\mu\Big)^{1/t},
\end{align*}
where
\[
\Big(\,\vint{B(x^{*},2^{-m+1})}g_j^{t}\,d\mu\Big)^{1/t}
\le C\Big(\,\vint{9B_{x}}g_j^{t}\,d\mu\Big)^{1/t}
\le C \big(\M g_{j}^{t}(x)\big)^{1/t}.
\]
Since
\[
2^{-m}\le100r(x)=10\operatorname{dist}(x,S)\le10d(x,y)<10\cdot2^{-k},
\]
we have that $m\ge k-4$, and hence
\begin{align*}
| \tilde{E}u(x) - m_{u}(B_{x^*}\cap S)|
&\le C2^{-ks}\sum_{j\ge k-6}2^{(k-j)(s-\eps')}\big(\M g_{j}^{t}(x)\big)^{1/t}\\
&\le Cd(x,y)^{s}\sum_{j\ge k-6}2^{(k-j)(s-\eps')}\big(\M g_{j}^{t}(x)\big)^{1/t}.
\end{align*}
Next we estimate the second term in \eqref{xV yF}.
Let $l$ be the smallest integer such that $B(y,2^{-k})\subset2^{l}B_{x^{*}}$.
Then $2^{l+1}B_{x^{*}}\subset B(x,17\cdot 2^{-k})$ and the radius of the ball $2^{l-1}B_{x^{*}}$ is at most $2^{-k+2}$.
Moreover, by the selection of $l$, the radius $2^{l-m}$ of the ball $2^{l}B_{x^{*}}$ is comparable to $2^{-k}$,
\begin{equation}\label{lmk}
2^{-k}\le 2^{-k}+d(x^{*},y)\le 2^{l-m}\le2^{-k+3}.
\end{equation}
We have
\begin{align*}
& |u(y) -m_{u}(B_{x^*}\cap S)|
\le |u(y)-m_{u}(B(y,2^{-k})\cap S)|\\
&\quad+|m_{u}(B(y,2^{-k})\cap S)-m_{u}(2^{l}B_{x^*}\cap S)|
+|m_{u}(2^{l}B_{x^*}\cap S)-m_{u}(B_{x^*}\cap S)|\\
&\quad=(a)+(b)+(c),
\end{align*}
and we estimate the terms $(a)-(c)$ separately.

Let $y$ be such that \eqref{median limit} holds (almost every point is such a point).
Using \eqref{estimate for medians} and estimating the geometric series in the third row by its first term $2^{-k\eps'}$,
we obtain
\begin{align*}
(a)
&\le\sum_{i=0}^{\infty}|m_{u}(B(y,2^{-i-k})\cap S)-m_{u}(B(y,2^{-(i+1)-k})\cap S)|\\
 &\le C\sum_{i=0}^{\infty}2^{(-i-k)\eps'}
\sum_{j=(i+k)-2}^{\infty} 2^{-j(s-\eps')}\Big(\,\vint{B(y,2^{-i-k+1})}g_{j}^{t}\,d\mu\Big)^{1/t}\\
 &\le C\sum_{j=k-2}^{\infty}2^{-j(s-\eps')}\big(\M g_{j}^{t}(y)\big)^{1/t}
\sum_{i=0}^{\infty}2^{(-i-k)\eps'}\\
&\le C2^{-k\eps'}\sum_{j=k-2}^{\infty}2^{-j(s-\eps')}\big(\M g_{j}^{t}(y)\big)^{1/t}\\
&\le C2^{-ks}\sum_{j=k-2}^{\infty}2^{(k-j)(s-\eps')}\big(\M g_{j}^{t}(y)\big)^{1/t}\\
&\le Cd(x,y)^{s}\sum_{j=k-2}^{\infty}2^{(k-j)(s-\eps')}\big(\M g_{j}^{t}(y)\big)^{1/t}.
\end{align*}
For term (b), we use \eqref{estimate for medians} and the fact
$B(y,2^{-k})\subset 2^{l}B_{x^{*}}\subset 5B(y,2^{-k})$
to obtain
\begin{align*}
(b)&=|m_{u}(B(y,2^{-k})\cap S)-m_{u}(2^{l}B_{x^*}\cap S)|\\
 &\le C 2^{-(m-l)\eps'}\sum_{j\ge m-l-2} 2^{-j(s-\eps')}
 \Big(\,\vint{2^{l+1}B_{x^{*}}}g_{j}^{t}\,d\mu\Big)^{1/t}.
\end{align*}
Now \eqref{lmk} together with the preceding discussion implies that
\[
2^{l+1}B_{x^{*}}\subset 17B(x,2^{-k})\subset C2^{l+1}B_{x^{*}},
\]
and hence
\begin{align*}
(b)
&\le C 2^{-k\eps'}\sum_{j\ge k-5} 2^{-j(s-\eps')}
  \Big(\,\vint{17B(x,2^{-k})}g_{j}^{t}\,d\mu\Big)^{1/t}\\
&\le C 2^{-ks}\sum_{j\ge k-5} 2^{(k-j)(s-\eps')}
 \big(\M g_{j}^{t}(x)\big)^{1/t}\\
&\le C d(x,y)^{s}\sum_{j\ge k-5} 2^{(k-j)(s-\eps')}
 \big(\M g_{j}^{t}(x)\big)^{1/t}.
\end{align*}
For the third term (c), we have, using similar estimates as above
\begin{align*}
(c)
&=|m_{u}(2^{l}B_{x^*}\cap S)-m_{u}(B_{x^*}\cap S)|\\
&\le\sum_{i=0}^{l-1}|m_{u}(2^{i}B_{x^*}\cap S)-m_{u}(2^{i+1}B_{x^*}\cap S)|\\
&\le C\sum_{i=0}^{l-1} 2^{-(m-i-1)\eps'}\sum_{j\ge m-i-3} 2^{-j(s-\eps')}
\Big(\,\vint{2^{i+2}B_{x^{*}}}g_{j}^{t}\,d\mu\Big)^{1/t}.
\end{align*}
Since $B_{x}\subset B_{x^{*}}\subset 5B_{x}$, estimating the sum by the $(l-1).$ term and using \eqref{lmk}, we have
\begin{align*}
(c)
&\le C\sum_{i=0}^{l-1} 2^{-(m-i-1)s}\sum_{j\ge m-i-3} 2^{(m-i-1-j)(s-\eps')}
 \big(\M g_{j}^{t}(x)\big)^{1/t}\\
 &\le C2^{-ks}\sum_{j\ge k-5} 2^{(k-j)(s-\eps')}
 \big(\M g_{j}^{t}(x)\big)^{1/t}\\
  &\le Cd(x,y)^{s}\sum_{j\ge k-5} 2^{(k-j)(s-\eps')}
 \big(\M g_{j}^{t}(x)\big)^{1/t}.
\end{align*}

\noindent {\sc Case 3:}
$x,y\in V\setminus S$,
$d(x,y)\ge \min\{\operatorname{dist}(x,S), \operatorname{dist}(y,S)\}$.

We begin with inequality
\begin{align*}
|\tilde{E}u(x)-\tilde{E}u(y)|
&\le|\tilde{E}u(x)-m_{u}(B_{x^*}\cap S)| +|\tilde{E}u(y) - m_{u}(B_{y^*}\cap S)|\\
&\quad+|m_{u}(B_{x^*}\cap S) - m_{u}(B_{y^*}\cap S)|\\
&=(1)+(2)+(3),
\end{align*}
where $y^{*}\in B_y\cap S$ and $B_{y^*}$ are chosen similarly as point $x^{*}$ and ball $B_{x^*}$ for $x$ in the beginning of 
case 2.
The radii of balls $B_{x^*}$ and $B_{y^*}$ are denoted by $2^{-m_{x}}$ and $2^{-m_{y}}$.

We may assume that $\operatorname{dist}(x,S)\le\operatorname{dist}(y,S)$. Then
\[
r(y)=\tfrac1{10}\operatorname{dist}(y,S)
\le\tfrac1{10}(d(x,y)+\operatorname{dist}(x,S))\le\tfrac15d(x,y),
\]
and hence $\operatorname{dist}(y,S)\le 2d(x,y)$, $d(x,x^{*})<2^{-k+1}$ and
$d(y,y^{*})<3\cdot2^{-k}$.
Hence estimates for $(1)$ and $(2)$ follow similarly as for the first term of \eqref{xV yF}.

For the last term $(3)$, let $K\ge0$ be the smallest integer such that the radius of the ball
$2^{K}B_{y^{*}}$ is at least $2^{-k}$, that is, $2^{-k}\le 2^{K-m_{y}}<2^{-k+1}$. Then
\begin{align*}
&|m_{u}(B_{x^*}\cap S) - m_{u}(B_{y^*}\cap S)|\\
&\quad\le|m_{u}(B_{y^*}\cap S) - m_{u}(2^{K}B_{y^*}\cap S)|
+|m_{u}(2^{K}B_{y^*}\cap S) - m_{u}(B_{x^*}\cap S)|\\
&\quad=(\alpha)+(\beta).
\end{align*}
We begin with $(\alpha)$. If $K=0$, then $(\alpha)=0$. If $K>0$, then the radius of $2^{K-1}B_{y^*}$ is at most $2^{-k}$.
This together with the fact that  $d(y,y^{*})<3\cdot2^{-k}$ implies that
\[
2^{K+1}B_{y^*}\subset B(y,7\cdot2^{-k})\subset5\cdot2^{K+1}B_{y^{*}}.
\]
Hence, using \eqref{estimate for medians}, we have
\begin{align*}
(\alpha)
&\le\sum_{i=0}^{K-1}|m_{u}(2^{i}B_{y^*}\cap S)-m_{u}(2^{i+1}B_{y^*}\cap S)|\\
&\le C\sum_{i=0}^{K-1} 2^{-(m_{y}-i-1)\eps'}\sum_{j\ge m_{y}-i-3} 2^{-j(s-\eps')}
 \Big(\,\vint{2^{i+2}B_{y^{*}}}g_{j}^{t}\,d\mu\Big)^{1/t}\\
  &\le C\sum_{i=0}^{K-1} 2^{-(m_{y}-i-1)\eps'}\sum_{j\ge m_{y}-i-3} 2^{-j(s-\eps')}
 \big(\M g_{j}^{t}(y)\big)^{1/t}.
\end{align*}
As in the case 2 (c), we estimate the sum by the $(K-1).$ term and use the fact that
\begin{equation}\label{Kmk}
 2^{K-1-m_{y}}<2^{-k}\le2^{K-m_{y}}
\end{equation}
and obtain
\[
(\alpha)
\le Cd(x,y)^{s}\sum_{j\ge k-1} 2^{(k-j)(s-\eps')}
 \big(\M g_{j}^{t}(y)\big)^{1/t}.
\]
For $(\beta)$, let $L\ge0$ be the smallest  integer such that $2^{K}B_{y^*}\subset 2^{L}B_{x^*}$.
Now, by the selection of $L$ and \eqref{Kmk},
\begin{equation}\label{Lmk}
 2^{-k}\le2^{L-m_{x}}<2^{-k+4},
\end{equation}
 and hence $2^{L}B_{x^*}\subset22\cdot2^{K}B_{y^*}$, $2^{L}B_{x^*}\subset B(x,2^{-k})$
and $B(x,2^{-k})\subset 3\cdot2^{L}B_{x^*}$. Now
\begin{align*}
(\beta)
&=|m_{u}(2^{K}B_{y^*}\cap S) - m_{u}(B_{x^*}\cap S)|\\
&\le|m_{u}(2^{K}B_{y^*}\cap S) - m_{u}(2^{L}B_{x^*}\cap S)|
+|m_{u}(2^{L}B_{x^*}\cap S) - m_{u}(B_{x^*}\cap S)|,
\end{align*}
where, by similar estimates as above and \eqref{Lmk},
\[
|m_{u}(2^{K}B_{y^*}\cap S) - m_{u}(2^{L}B_{x^*}\cap S)|
\le Cd(x,y)^{s}\sum_{j\ge k-6} 2^{(k-j)(s-\eps')}
 \big(\M g_{j}^{t}(x)\big)^{1/t}.
\]
Similarly as for $(\alpha)$ above, we obtain
\[
|m_{u}(2^{L}B_{x^*}\cap S) - m_{u}(B_{x^*}\cap S)|
\le Cd(x,y)^{s}\sum_{j\ge k-6} 2^{(k-j)(s-\eps')}
 \big(\M g_{j}^{t}(x)\big)^{1/t}.
\]
\noindent
{\sc Case 4:}
$x,y\in V\setminus S$,
$d(x,y)< \min\{\operatorname{dist}(x,S), \operatorname{dist}(y,S)\}$.

We may assume that $\operatorname{dist}(x,S)\le\operatorname{dist} (y,S)$.
By the properties of the functions $\ph_{i}$ and the fact that $B_{i}\subset B_{x^{*}}$ with comparable measures whenever
$i\in  I_{x}\cup I_{y}$, we can use similar estimates as for the first term of \eqref{xV yF} and obtain
\begin{equation}\label{eux-euy}
 \begin{aligned}
|\tilde{E}u(x)-\tilde{E}u(y)|
&=\Big|\sum_{i\in I_{x}\cup I_{y}}(\ph_{i}(x)-\ph_{i}(y))
  \big(m_{u}(B^{*}_{i}\cap S) - m_{u}(B_{x^*}\cap S)\big)\Big|\\
 &\le C\frac{d(x,y)}{r(x)}
 2^{-m_{x}\eps'}\sum_{j\ge m_{x}-2}2^{-j(s-\eps')}\big(\M g_{j}^{t}(x)\big)^{1/t}.
\end{aligned}
\end{equation}
Using the assumptions $0<\delta<1-s$, $r(x)<2^{-m_{x}}$,
$d(x,y)<\operatorname{dist}(x,S)=10r(x)$ and $d(x,y)<2^{-k}$, we have
\begin{align*}
 d(x,y)r(x)^{-1}2^{-m_{x}\eps'}
&\le Cd(x,y)r(x)^{s+\delta-1} \ 2^{m_{x}(s-\eps'+\delta)}\\
&\le Cd(x,y)^{s+\delta} \ 2^{m_{x}(s-\eps'+\delta)}\\
&\le Cd(x,y)^{s}2^{(m_{x}-k)\delta+m_{x}(s-\eps')}.
\end{align*}
This together with \eqref{eux-euy} implies that
\[
|\tilde{E}u(x)-\tilde{E}u(y)|
\le Cd(x,y)^{s}
\sum_{j=m_{x}-2}^\infty 2^{(m_{x}-k)\delta+(m_{x}-j)(s-\eps')}\big(\M g_j^t(x)\big)^{1/t}.
\]
By splitting the sum in two parts and using the facts $m_{x}\le j+2$ and $m_{x}\le k$, we obtain
\begin{align*}
& \sum_{j=m_{x}-2}^\infty 2^{(m_{x}-k)\delta+(m_{x}-j)(s-\eps')}\big(\M g_j^t(x)\big)^{1/t}\\
=\ & \sum_{j=m_{x}-2}^{k-1} 2^{(m_{x}-k)\delta+(m_{x}-j)(s-\eps')}\big(\M g_j^t(x)\big)^{1/t}\\
&\quad+\sum_{j=k}^\infty 2^{(m_{x}-k)\delta+(m_{x}-j)(s-\eps')}\big(\M g_j^t(x)\big)^{1/t}\\
\le \ &C\bigg(\sum_{j=-\infty}^{k-1}2^{(j-k)\delta}\big(\M g_j^t(x)\big)^{1/t} \ + \
\sum_{j=k}^{\infty}2^{(k-j)(s-\eps')}\big(\M g_j^t(x)\big)^{1/t}\bigg),
\end{align*}
which implies the claim in case 4.
Cases 1-4 show that $(C\tilde g_k)$ is a fractional $s$-gradient for the extension $\tilde Eu$.
\end{proof}

Next will estimate the norm of the fractional $s$-gradient of the local extension.
Recall from \eqref{gradientOfextension} that
\[
\tilde g_{k}(x)=\sum_{j=-\infty}^{k-1}2^{(j-k)\delta}\big(\M g_j^t(x)\big)^{1/t} \ +
\sum_{j=k-6}^{\infty}2^{(k-j)(s-\eps')}\big(\M g_j^t(x)\big)^{1/t},
\]
where $0<\delta<1-s$, $0<\eps'<s$ and $0<t<\min\{p,q\}$.

\begin{lemma}\label{IntForGradOfExtension}
$\|(\tilde g_k)\|_{L^p(V,\,l^q)}\le C\|(g_k)\|_{L^p(S,\,l^q)}$.
\end{lemma}

\begin{proof}
We estimate only the $L^p(V,\,l^q)$ norm of
\[
\bigg(\sum_{j=-\infty}^{k-1}2^{(j-k)\delta}\big(\M g_j^t(x)\big)^{1/t}\bigg)_{k\in \z},
\]
since another part can be estimated in a similar way.

Lemma \ref{summing lemma} implies that
\[
\sum_{k\in\z}\bigg(\sum_{j=-\infty}^{k-1}2^{(j-k)\delta}\big(\M g_j^t(x)\big)^{1/t}\bigg)^q
\le C\sum_{j\in\z}\big(\M g_j^t(x)\big)^{q/t}
\]
Hence, using a version of the Fefferman--Stein vector valued maximal function theorem for metric space with a doubling 
measure, proved in \cite[Thm 1.3]{Sa}, \cite[Thm 1.2]{GLY} (for the original version, see \cite[Thm 1]{FS}), we obtain
\begin{align*}
\|(\tilde g_k)_{k\in\z}\|_{L^p(V,\,l^q)}
&\le\|(M g_k^t)_{k\in\z}\|_{L^{p/t}(V,\,l^{q/t})}^{1/t}
\le C\|(g_k^t)_{k\in\z}\|_{L^{p/t}(X,\,l^{q/t})}^{1/t}\\
&=C\|\vec{g}\|_{L^{p}(X,l^{q})}=C\|\vec{g}\|_{L^{p}(S,\,l^{q})},
\end{align*}
where the last equality holds, since $\vec{g}\equiv 0$ outside $S$.
\end{proof}

If $u\in N^s_{p,q}(S)$ and $(g_{k})_{k\in\z}\in \D^{s}(u)$ with
\[
\|(g_{k})\|_{l^q(L^p(S))}<2\inf_{{(h_{k})\in \D^{s}(u)}}\|(h_{k})\|_{l^q(L^p(S))},
\]
then we proceed as in the Triebel--Lizorkin case. Lemma \ref{lemmaGradient} gives a fractional $s$-gradient $(\tilde g_k)_{k
\in\z}$ for the local extension, and the norm estimate corresponding Lemma \ref{IntForGradOfExtension} follows from the 
lemma below.

\begin{lemma}\label{IntForGradOfExtensionBesov}
$\|(\tilde g_k)\|_{l^q(L^p(V))}\le C\|(g_k)\|_{l^q(L^p(S))}$.
\end{lemma}
\begin{proof}
As in the proof of Lemma \ref{IntForGradOfExtension}, we estimate the first part of $(\tilde g_k)$ only.
The second part can be estimated similarly.

Assume first that $p\ge 1$.
By the Minkowski inequality and the boundedness of the Hardy--Littlewood maximal operator in $L^r$, $r>1$, we have
\[
\begin{split}
\Big\|\sum_{j=-\infty}^{k}2^{(j-k)\delta}\left(\M g_j^t\right)^{1/t}\Big\|_{L^p(V)}
&\le \sum_{j=-\infty}^{k}2^{(j-k)\delta}\big\|\left(\M g_j^t\right)^{1/t}\big\|_{L^p(V)}\\
&\le \sum_{j=-\infty}^{k}2^{(j-k)\delta}\|g_j\|_{L^p(X)},
\end{split}
\]
and Lemma \ref{summing lemma} implies that
\[
\sum_{k\in\z}\Big(\sum_{j=-\infty}^{k}2^{(j-k)\delta}\|g_j\|_{L^p(V)}\Big)^q
\le C\sum_{j\in\z}\|g_j\|_{L^p(X)}^q.
\]

Assume then that $0<p<1$. Using inequality \eqref{a sum}
and the boundedness of the Hardy--Littlewood maximal operator in $L^r$, $r>1$, we obtain
\[
\begin{split}
\Big\|\sum_{j=-\infty}^{k}2^{(j-k)\delta}\left(\M g_j^t\right)^{1/t}\Big\|_{L^p(V)}^p
&\le \sum_{j=-\infty}^{k}2^{(j-k)\delta p}\big\|\left(\M g_j^t\right)^{1/t}\big\|_{L^p(V)}^p\\
&\le \sum_{j=-\infty}^{k}2^{(j-k)\delta p}\|g_j\|_{L^p(X)}^p.
\end{split}
\]
Hence, using Lemma \ref{summing lemma}, we have
\[
\begin{split}
\sum_{k\in\z} \Big\| \sum_{j=-\infty}^{k}2^{(j-k)\delta}\Big(\M g_j^t\Big)^{1/t}\Big\|_{L^p(V)}^{q}
&\le \sum_{k\in\z}\Big(\sum_{j=-\infty}^{k}2^{(j-k)\delta p}\|g_j\|_{L^p(X)}^p\Big)^{q/p} \\
&\le C\sum_{j\in\z}\|g_j\|_{L^p(X)}^q.
\end{split}
\]
The desired norm estimate follows in both cases because $\vec{g}\equiv 0$ outside $S$.
\end{proof}

\noindent
{\bf The final extension:}

Now we are ready to define the final extension.
Let $\Psi\colon X\to [0,1]$ be an $L$-Lipschitz cut-off function such that $\Psi|_{S}=1$ and $\Psi|_{X\setminus V}=0$.
We define an extension operator by setting
\[
Eu = \Psi \tilde{E}u.
\]
Then $Eu=u$ in $S$ and, by \eqref{Eu norm},
\[
\|Eu\|_{L^{p}(X)}
\le \|\tilde Eu\|_{L^{p}(V)}\le C\|u\|_{L^{p}(S)}.
\]
In the Triebel--Lizorkin case, \eqref{Eu norm} together with Lemmas \ref{lemmaGradient} and
\ref{IntForGradOfExtension} imply that $\tilde Eu\in M^{s}_{p,q}(V)$ and
$\|(\tilde g_k)\|_{L^p(V,\,l^q)}\le\|(g_k)\|_{L^p(S,\,l^q)}$.
Now, by Lemma \ref{LemWithLipForTL}, the sequence $(g'_{k})_{k\in\z}$,
\[
g'_k=
\begin{cases}
(\tilde g_k+2^{s k+2}|\tilde Eu|)\ch{\operatorname{supp}\Psi},
  \quad &\text{if}\,\,k<k_L,\\
(\tilde g_k+2^{k(s-1)}L|\tilde Eu|)\ch{\operatorname{supp}\Psi},
  \quad &\text{if}\,\, k\ge k_L,
\end{cases}
\]
where $k_L$ is the integer such that $2^{k_L-1}<L\le 2^{k_L}$,
is a fractional $s$-gradient of $Eu$ and it satisfies norm estimate
\[
\|\vec{g'}\|_{L^p(X,\,l^q)}
\le C\|\tilde Eu\|_{M^{s}_{p,q}(V)}
\le C\|u\|_{M^{s}_{p,q}(S)}.
\]
Hence $Eu\in M^{s}_{p,q}(X)$ and $\|Eu\|_{M^{s}_{p,q}(X)}\le C \|u\|_{M^{s}_{p,q}(S)}$.
The Besov case follows similarly by using Lemma \ref{IntForGradOfExtension} instead of Lemma \ref{IntForGradOfExtension} 
and Remark \ref{WithLipForB}.

This concludes the proof of Theorem \ref{m extension}.\qed

\section{Measure density from extension}\label{section: measure density from extension}
The main theorem of this section shows that if the space $X$ is $Q$-regular and geodesic, then the measure density of a 
domain is a necessary condition for the extension property of functions from Haj\l asz--Besov and Haj\l asz--Triebel--Lizorkin 
spaces.
The analogous result for functions from Haj\l asz--Sobolev spaces was proved in \cite[Thm 5]{HKT_Revista}, and the proof 
given below is
inspired by the corresponding proof in that paper. The main tools in the proof are Lipschitz estimates from Section \ref{section: 
lip} and embedding theorems, both the old ones and Theorem \ref{thm: Lorentz-Sobolev}. The assumption that $X$ is 
geodesic is used only to get the property $\mu(\partial B)=0$ for all balls $B$.

\begin{theorem}\label{measure density from extension}
Let $X$ be a $Q$-regular, geodesic metric measure space.
Let $0<s<1$, $0<p<\infty$, and $0<q\le\infty$.
If $\Omega\subset X$ is an $M^s_{p,q}$-extension domain (or an $N^s_{p,q}$-extension domain),
then it satisfies measure density condition \eqref{measure density}.
\end{theorem}

\begin{proof}
First we assume that $\Omega$ is an $M^s_{p,q}$-extension domain for some $0<s<1$, $0<p<\infty$, and $0<q\le\infty$.
To show that the measure density condition holds, let $x\in\overline{\Omega}$ and $0<r\le 1$, and let $B=B(x,r)$.
We may assume that $\Omega\setminus B(x,r)\ne\emptyset$, otherwise
the measure density condition is obviously satisfied.

We split the proof into three different cases depending on the size of $sp$.
\smallskip

\noindent {\sc Case 1:}  $0< s p<Q$.
By the proof of \cite[Proposition 13]{HKT_Revista}, the geodesity of $X$ implies that $\mu(\partial B(x,R))=0$
for every $R>0$. Hence there exist radii $0<\tilde{\tilde{r}}<\tilde{r}<r$ such that
\[
\mu(B(x,\tilde{\tilde{r}})\cap\Omega)
=\tfrac{1}{2}\mu(B(x,\tilde{r})\cap\Omega)
=\tfrac{1}{4}\mu(B(x,r)\cap\Omega).
\]
Let $u\colon\Omega\to[0,1]$,
\begin{equation*}
u(y)=
\begin{cases}
1, &\text{if }y\in B(x,\tilde{\tilde{r}})\cap\Omega,\\
\frac{\tilde{r}-d(x,y)}{\tilde{r}-\tilde{\tilde{r}}}, &\text{if }y\in B(x,\tilde{r})\setminus B(x,\tilde{\tilde{r}})\cap\Omega,\\
0, &\text{if }y\in \Omega\setminus
B(x,\tilde{r}).
\end{cases}
\end{equation*}
Since the function $u$ is $1/(\tilde{r}-\tilde{\tilde{r}})$-Lipschitz and $\|u\|_\infty\le 1$,
by Corollary \ref{lemma:lip frac gradient} and the fact that $0<\tilde r-\tilde{\tilde r}<1$, we have
\begin{equation}\label{Msps norm}
\| u\|_{M^s_{p,q}(\Omega)}
\le C \mu(B(x,\tilde{r})\cap\Omega)^{1/p}(\tilde{r}-\tilde{\tilde{r}})^{-s}.
\end{equation}
We want to find a good lower bound for $\| u\|_{M^s_{p,q}(\Omega)}$.
Let $v\in M^s_{p,q}(X)$ be an extension of $u$, and let
$(h_k)_{k\in\z}\in \D^s(v)$ be such that
\[
\|(h_k)_{k\in\z}\|_{L^p(X,\,l^q)}\le C\| v\|_{M^s_{p,q}(X)}.
\]
Since
\[
|v(z)-v(y)|\le C d(z,y)^{s}\big(\sup_{k\in\z}h_k(z)+\sup_{k\in\z}h_k(y)\big)
\]
for almost every $z,y\in X$, we have that $v\in M^{s,p}(X)$ with an $s$-gradient $H=\sup\limits_{k\in\z}h_k$.
Now, by Lemma \ref{lemma:Sobolev-Poincare 1},
\begin{equation}\label{Poincare}
\inf_{c\in\re}\Big(\,\vint{B(x,1)}|v-c|^{p^*(s)}\,d\mu\Big)^{1/p^*(s)}
\le C\Big(\,\vint{B(x,2)}H^{p}\,d\mu\Big)^{1/p},
\end{equation}
where $p^*(s)=Qp/(Q-ps)$ and, by the selection of $(h_k)_{k\in\z}$,
\[
\Big(\int_{B(x,2)}H^{p}\,d\mu \Big)^{1/p}
\le \Big(\int_{X}\sup_{k\in\z}\,h_k^p\,d\mu\Big)^{1/p}
\le C\Vert v\Vert_{M^{s}_{p,q}(X)}.
\]
Inequality \eqref{Poincare}, together with the last estimate, the $Q$-regularity and the fact that $v$ is an extension of $u$, 
implies the existence of $c_0$ such that
\[
\Big(\int_{B(x,1)}|v-c_{0}|^{p^*(s)}\,d\mu\Big)^{1/p^*(s)}
\le C\Vert v\Vert_{M^{s}_{p,q}(X)}
\le C\Vert u\Vert_{M^{s}_{p,q}(\Omega)}.
\]
Hence, by the Chebyshev inequality, we have
\begin{equation}\label{Cheb}
 \big(\mu\big(\{y\in B(x,1):\,|v(y)-c_0|>\lambda\}\big)\big)^{1/p^*(s)}
\le \frac{C}{\lambda}\Vert u\Vert_{M^{s}_{p,q}(\Omega)}.
\end{equation}
Since $u=v=1$ on $B(x,\tilde{\tilde{r}})\cap\Omega$ and $u=v=0$ on $(B(x,r)\setminus B(x,\tilde{r}))\cap\Omega$, we have 
that
\[
|v-c_0|\geq 1/2
\]
on at least one of the sets $B(x,\tilde{\tilde{r}})\cap\Omega$ and $(B(x,r)\setminus B(x,\tilde{r}))\cap\Omega$.
Since the two sets have measures comparable to the measure of $B(x,\tilde{r})\cap\Omega$,
\eqref{Cheb} with $\lambda=1/2$ gives
\[
\mu(B(x,\tilde{r})\cap\Omega)^{1/p^*(s)}\le C\Vert u\Vert_{M^{s}_{p,q}(\Omega)}.
\]
This together with \eqref{Msps norm} shows that
\[
\mu(B(x,\tilde{r})\cap\Omega)^{1/p^*(s)}
\le C(\tilde{r}-\tilde{\tilde{r}})^{-s}\mu(B(x,\tilde{r})\cap\Omega)^{1/p},
\]
and hence
\begin{equation}\label{difr}
\tilde{r}-\tilde{\tilde{r}}\le C \mu(B(x,\tilde{r})\cap\Omega)^{1/Q}.
\end{equation}
Now, defining radii $r_j$, $j=0,1,\dots$, as
\[
r_0=r,\quad \quad r_{j+1}=\tilde{r_j},
\]
we have
\[
\mu(B(x,r_j)\cap\Omega)=2^{-j}\mu(B(x,r)\cap\Omega),
\]
which implies that $r_j\to 0$ as $j\to\infty$.
Applying inequality \eqref{difr} for $r_{j+1}$, we obtain
\[
r_{j+1}-r_{j+2}
\le C \mu(B(x,r_{j+1})\cap\Omega)^{1/Q}\le C2^{-j/Q}\mu(B(x,r)\cap\Omega)^{1/Q},
\]
and hence
\[
\tilde{r}=r_1=\sum_{j=0}^\infty (r_{j+1}-r_{j+2})\le C\mu(B(x,r)\cap\Omega)^{1/Q}.
\]
Since it was proved in \cite[Lemma 14]{HKT_Revista}, that if measure density condition
\eqref{measure density} holds for all $x\in\overline{\Omega}$ and all $0<r\le 1$
such that $r\le 10\tilde{r}$, it holds for all $x\in\overline{\Omega}$ and all $0<r\le 1$, we are done in this case.
Note that the assumption of connectedness of $\Omega$ is essential in the cited lemma.
\medskip

\noindent {\sc Case 2:} $sp>Q$.
Let $u\colon\Omega\to[0,1]$,
\[
u(y)=
\begin{cases}
1-\frac{d(x,y)}{r},  &\text{if }y\in B(x,r),\\
0, &\text{if }y\in \Omega\setminus B(x,r).
\end{cases}
\]
Since the function $u$ is $r^{-1}$-Lipschitz and $\|u\|_{\infty}\le1$, using Corollary \ref{lemma:lip frac gradient} and the fact 
that $0<r<1$, we obtain
\begin{equation}\label{u norm big sp}
\Vert u\Vert_{M^s_{p,q}(\Omega)}
\le C\big(\mu(B(x,r)\cap\Omega)\big)^{1/p}r^{-s}.
\end{equation}
Let $v\in M^{s}_{p,q}(X)$ be an extension of $u$, and let $(h_k)_{k\in\z}\in \D^s(v)$ be such that
\begin{equation}\label{h norm}
\|(h_k)_{k\in\z}\|_{L^p(X,\,l^q)}
\le C\| v\|_{M^s_{p,q}(X)}
\le C\| u\|_{M^s_{p,q}(\Omega)} .
\end{equation}
As in the case $sp<Q$, since $v\in M^{s}_{p,q}(X)$ and $(h_k)_{k\in\z}$ is its fractional $s$-gradient, we have that
\[
|v(z)-v(y)|\le C d(z,y)^{s}\big(\sup_{k\in\z}h_k(z)+\sup_{k\in\z}h_k(y)\big)
\]
for almost every $z,y\in X$.
Thus $v\in M^{s,p}(X)$ and $H=\sup\limits_{k\in\z}h_k$ is its $s$-gradient.
By the analogue of \cite[Lemma 8]{H2} (the proof goes in the same way),
\begin{equation}\label{AlphaGradient}
|v(x)-v(y)|\le C r^{Q/p}d(x,y)^{s-Q/p}\Big(\,\vint{B(x,5r)}H^p\,d\mu\Big)^{1/p}.
\end{equation}
Since $v(x)=u(x)=1$ and $v(y)=u(y)=0$ for some $y\in (\Omega\setminus B(x,r))\cap B(x,2r)$
(we can assume that \eqref{AlphaGradient} holds for these particular points $x$ and $y$),
using \eqref{h norm}, \eqref{u norm big sp} and the $Q$-regularity, we obtain
\[
\begin{split}
1&\le C r^{Q/p}r^{s-Q/p}\Big(\,\vint{B(x,5r)}H^p\,d\mu\Big)^{1/p}
\le C r^{s-Q/p}\|(h_k)_{k\in\z}\|_{L^p(X,\,l^q)}\\
&\le  Cr^{-Q/p}\mu(B(x,r)\cap\Omega)^{1/p},
\end{split}
\]
which implies the measure density by the $Q$-regularity.
\medskip

\noindent {\sc Case 3:} $sp=Q$.
We will use the following modification of \cite[Thm 5.9]{HK}.
Below, $\cH^1_\infty$ is the Hausdorff content of dimension $1$.
\begin{lemma}\label{LemHeinKosk}
Let $X$ be a $Q$-regular space, $Q\geq 1$. Let $E$ and $F$ be
disjoint subsets of a ball $B=B(x,r)$ such that
\begin{equation}\label{lemmaHeKo}
\min\{\cH^1_\infty(E),\cH^1_\infty(F)\}\ge \lambda r,
\end{equation}
for some $0<\lambda\le 1$. Then there is a constant $C\geq 1$, depending only on $X$, such that
\[
\int_{20B}g^p\,d\mu\geq C\lambda
\]
whenever $u$ is locally integrable, $g$ is a $(Q/p)$-gradient of $u$ in $20B$,
every point in $E\cup F$ is a Lebesgue point of $u$, $u|_E\geq 1$ and $u|_F\le 0$.
\end{lemma}

Let $B=B(x,r)$ and let $A=\tfrac{2}{3}B\setminus \tfrac{1}{3}B$.
Let $u\colon\Omega\to[0,1]$,
\[
u(y)=
\begin{cases}
1,  &\text{if }y\in \frac{1}{3}B\cap\Omega,\\
2-\frac{3d(x,y)}{r},&\text{if } y\in A\cap\Omega,\\
0, &\text{if }x\in \Omega\setminus\frac{2}{3}B.
\end{cases}
\]
The function $u$ is $3/r$-Lipschitz and, as above, by Corollary \ref{lemma:lip frac gradient}, we obtain
\[
\Vert u\Vert_{M^s_{p,q}(\Omega)}\le C
(\mu(B(x,r)\cap\Omega))^{1/p}r^{-s}.
\]
Let $v\in M^{s}_{p,q}(X)$  be an extension of $u$ and let $(h_k)_{k\in\z}\in \D^s(v)$ such that
\[
\|(h_k)_{k\in\z}\|_{L^p(X,\,l^q)}\le C\Vert
v\Vert_{M^s_{p,q}(X)}\le C\Vert u\Vert_{M^s_{p,q}(\Omega)}.
\]
Using the connectivity of $\Omega$ and the fact that the $1$-Lipschitz function $y\mapsto d(x,y)$ does not increase the 
Hausdorff $1$-content, we obtain, as in \cite[p.665]{HKT_Revista}, that
\[
\min\big\{\cH^1_\infty(\tfrac{1}{3}B\cap\Omega),\cH^1_\infty(B\setminus\tfrac{2}{3}B)\cap \Omega)\big\}
\ge\frac{r}{3}.
\]
Applying Lemma \ref{LemHeinKosk} to the function $v$ with a $(Q/p)$-gradient $H=\sup_{k\in\z}h_k$, we obtain
\[
C\le\int_{20B}H^p\,d\mu\le\mu(B\cap\Omega)r^{-Q},
\]
which implies the measure density by the $Q$-regularity.

We have shown that the extension property for Triebel--Lizorkin spaces implies the measure density condition for a domain.
\medskip

In order to obtain the analogous result for Besov spaces, we have to make the following modifications in the proof given 
above.
Observe that in all three cases,
the $M^s_{p,q}(\Omega)$-norms and $N^s_{p,q}(\Omega)$-norms of the chosen test functions have the same upper bounds,
see Lemma \ref{lemma:lip frac gradient}.
\smallskip

\noindent {\sc Case 1:}  $0<sp<Q.$ Instead of \eqref{Cheb}, we use an estimate

\begin{equation*}
 \big(\mu\big(\{y\in X:\,|v(y)-c_0|>\lambda\}\big)\big)^{1/p^*(s)}
\le \frac{C}{\lambda}\| u\|_{N^{s}_{p,q}(X)},
\end{equation*}
which follows from the case $q=\infty$ of Theorem \ref{thm: Lorentz-Sobolev} and the fact that
$\| u\|_{N^{s}_{p,\infty}(X)}\le \| u\|_{N^{s}_{p,q}(X)}$ for $0<q\le\infty$.
\medskip

\noindent {\sc Case 2:} $sp>Q$. Instead of \eqref{AlphaGradient}, we use the following lemma.

\begin{lemma}
Let $X$ be a $Q$-regular space, $Q\ge 1$. Let $0<s<1$ and $sp>Q$.
There is a constant $C>0$, such that for each $u\in\dot N^{s}_{p,q}(X)$,
\[
|u(x)-u(y)|\le Cd(x,y)^{s-Q/p}\|u\|_{\dot N^{s}_{p,q}(X)}
\]
for almost every $x,y\in X$.
\end{lemma}
\begin{proof}
Using Poincar\'e inequality \eqref{eq: Poincare}, the H\"older inequality and the $Q$-regularity, we obtain
\[
\vint{B(x,r)}|u-u_{B(x,r)}|\,d\mu
\le  Cr^{s-Q/p}\|u\|_{\dot N^{s}_{p,q}(X)},
\]
for every $x\in X$ and $r>0$. The claim follows now from \cite[Thm 4]{MS2}.
\end{proof}

\noindent {\sc Case 3:}  $sp=Q$. Let the radii $\tilde{\tilde{r}}$ and $\tilde{r}$  and the function $u$ be as in Case 1.
Let $v\in N^s_{p,q}(X)$ be an extension of $u$ with $\|v\|_{N^s_{p,q}(X)}\le C\|u\|_{N^s_{p,q}(\Omega)}$.
By Poincar\'e inequality \eqref{eq: Poincare}, the H\"older inequality and the $Q$-regularity, $v\in\operatorname{BMO}(X)$ and
\[
\|v\|_{\operatorname{BMO}(X)}\le C\|v\|_{N^s_{p,q}(X)}.
\]
Hence, by the John--Nirenberg theorem \cite[Thm 2.2]{B},
\begin{equation}\label{eq: John-Nirenberg}
\inf_{b\in\re}\ \vint{B(x,r)}\exp\bigg(\frac{|v-b|}{C\|v\|_{N^s_{p,q}(X)}  }\bigg)\le C.
\end{equation}
By Corollary \ref{lemma:lip frac gradient}, we have
\begin{equation}\label{eq: norm of extension}
\|v\|_{N^s_{p,q}(X)}\le C\|u\|_{N^s_{p,q}(\Omega)}\le C(\tilde r-\tilde{\tilde r})^{-s}\mu(B(x,\tilde r)\cap\Omega).
\end{equation}

Since $u=v=1$ on $B(x,\tilde{\tilde{r}})\cap\Omega$ and $u=v=0$ on $(B(x,r)\setminus B(x,\tilde{r}))\cap\Omega$,
we have that $|v-c_0|\geq 1/2$ on at least one of the sets $B(x,\tilde{\tilde{r}})\cap\Omega$ and
$(B(x,r)\setminus B(x,\tilde{r}))\cap\Omega$.
Since the measures of these two sets are comparable to the measure of $B(x,\tilde{r})\cap\Omega$,
\eqref{eq: John-Nirenberg} and \eqref{eq: norm of extension} imply
\[
\mu(B(x,\tilde r)\cap\Omega)\exp\big(C^{-1}(\tilde r-\tilde{\tilde r})^s\mu(B(x,\tilde r)\cap\Omega)^{-1/p} \big)
\le C r^Q,
\]
which can be written in the form
\[
\tilde r-\tilde{\tilde r}
\le C\mu(B(x,\tilde r)\cap\Omega)^{1/Q}\log\bigg(\frac{C r^Q}{\mu(B(x,\tilde r)\cap\Omega)}\bigg)^{1/s}.
\]
Now, defining radii $r_j$, $j=0,1,\dots$, as
\[
r_0=r,\quad \quad r_{j+1}=\tilde{r_j},
\]
we have
\[
\mu(B(x,r_{j+1})\cap\Omega)=2^{-j}\mu(B(x,\tilde r)\cap\Omega),
\]
which implies
\[
\tilde r=r_1
=\sum_{j=0}^\infty(r_{i+1}-r_{i+2})\le C\mu(B(x,\tilde r)\cap\Omega)^{1/Q}
\sum_{j=0}^\infty 2^{-j/Q}\log\bigg(\frac{C2^{j} r_j^Q}{\mu(B(x,\tilde r)\cap\Omega)}\bigg)^{1/s}.
\]
A similar argument as in \cite[p.\,1228--1229]{HKT_JFA}, shows that
\[
\sum_{j=0}^\infty 2^{-j/Q}\log\bigg(\frac{C2^{j} r_j^Q}{\mu(B(x,\tilde r)\cap\Omega)}\bigg)^{1/s}\le C.
\]
Thus, the measure density condition holds for all $x\in\overline\Omega$ and $0<r\le1$ such that $r\le 10 \tilde r$,
and the claim follows by \cite[Lemma 14]{HKT_Revista}, which tells that it suffices to prove \eqref{measure density} in that 
case.
\end{proof}

\begin{remark} The proof of Case 3 for Besov spaces  is a modification of the proof of \cite[Thm 1.b)]{HKT_JFA};
if $q<\infty$, one could simplify the reasoning using the proof of Case 3 for Triebel--Lizorkin spaces
with Lemma \ref{LemHeinKosk} replaced by \cite[Lemma 3.3]{KKSS}.
\end{remark}

\begin{remark} 
Since Haj\l asz--Besov and Haj\l asz--Triebel--Lizorkin functions do not see the sets of measure zero, it is also natural to discuss a connection between the extension property and an "almost everywhere" variant of the measure density condition. Indeed, the proof of Theorem \ref{measure density from extension} shows that
if $S\subset X$ is connected and there exists a bounded extension operator $E\colon N^s_{p,q}(S)\to N^s_{p,q}(X)$ (or 
$E\colon M^s_{p,q}(S)\to M^s_{p,q}(X)$), then the set 
\[
\tilde S=\{x\in X: \mu(B(x,r)\cap S)>0 \text{ for every } r>0\}
\] 
satisfies the measure density condition. 
Since $\mu( S\setminus \tilde S)=0$, it follows that, for a connected set $S$, a bounded extension operator exists if and only if \eqref{measure density} holds for almost every $x\in S$ and every $0<r\le 1$. 
\end{remark}


\section{Extension theorems for Besov and Triebel--Lizorkin spaces in $\rn$}\label{Rn}
In this section we apply our general results to obtain extension results for classical Besov and
Triebel--Lizorkin spaces defined in the Euclidean space.

\subsection{Besov spaces on subsets of $\rn$}
Let $S\subset\rn$ be a measurable set and let $t>0$. 
For $h\in\rn$, define $S-h=\{s-h:s\in S\}$ and $S_h=S\cap(S-h)$.
We consider the following versions of the $L^p$-modulus of smoothness on $S$:

\begin{equation}\label{modulus of smoothness on S}
\omega(u,S,t)_p=\sup_{|h|\le t}\Big(\int_{S_h}|u(x+h)-u(x)|^p\,dx\Big)^{1/p},
\end{equation}

\begin{equation}\label{averaged modulus of smoothness on S}
\begin{split}
E_p(u,S,t)&=\bigg(\,\vint{B(0,t)}\,\int_{S_h}|u(x+h)-u(x)|^p\,dx\,dh\bigg)^{1/p}\\
&=\Big(\int_S\,\frac{1}{|B(x,t)|}\int_{B(x,t)\cap S}|u(y)-u(x)|^p\,dy\,dx\Big)^{1/p}
\end{split}
\end{equation}
and
\begin{equation}\label{loc pol approx on S}
\hat E_p(u,S,t)
=\Big(\int_S\,\frac{1}{|B(x,t)|}\inf_{c\in\re}\int_{B(x,t)\cap S}|u(y)-c|^p\,dy\,dx\Big)^{1/p}.
\end{equation}
Versions \eqref{modulus of smoothness on S} and \eqref{averaged modulus of smoothness on S} were used, for example, in 
\cite{DS} and \eqref{loc pol approx on S} in \cite{Sh86}, \cite{ShRn}. 
Note that \eqref{averaged modulus of smoothness on S} and \eqref{loc pol approx on S} are connected to the smoothness 
functions $C^{s,r}_tu(x)$ and $I^{s,r}_tu(x)$ from \cite[Def. 1.1]{GKZ}, 
The Besov spaces $B^s_{p,q}(S)$, $\cB_{p,q}^s(S)$ and $\hat \cB_{p,q}^s(S)$  consist of measurable functions for which the 
norms
\[
\|u\|_{B^s_{p,q}(S)}
=\|u\|_{L^p(S)}+\bigg(\int_0^1\big(t^{-s}\omega_p(u,S,t)\big)^{q}\frac{dt}{t}\bigg)^{1/q},
\]
\[
\|u\|_{\cB^s_{p,q}(S)}
=\|u\|_{L^p(S)}+\bigg(\int_0^1\big(t^{-s}E_p(u,S,t)\big)^{q}\frac{dt}{t}\bigg)^{1/q}
\]
and
\[
\|u\|_{\hat\cB^s_{p,q}(S)}
=\|u\|_{L^p(S)}+\bigg(\int_0^1\big(t^{-s}\hat E_p(u,S,t)\big)^{q}\frac{dt}{t}\bigg)^{1/q}
\]
are finite respectively.

The following theorem describes how these spaces are related to each other, and to the Haj\l asz--Besov space $N^s_{p,q}$.

\begin{theorem}\label{relation of besov spaces2}
Let $0<s<1$, $0<p<\infty$ and $0<q\le\infty$. Then
\[
N^s_{p,q}(S)\subset B^s_{p,q}(S)\subset  \cB^s_{p,q}(S)\subset\hat \cB^s_{p,q}(S),
\]
and there is a constant $C>0$ such that
\begin{equation}\label{besov rn}
\|\cdot\|_{\hat \cB^s_{p,q}(S)}\le \|\cdot\|_{\cB^s_{p,q}(S)}
\le  \|\cdot\|_{B^s_{p,q}(S)}\le C\|\cdot\|_{N^s_{p,q}(S)},
\end{equation}
for each measurable set $S\subset\rn$.

If $S$ satisfies measure density condition \eqref{measure density}, then
\begin{equation}\label{besov spaces coincide}
N^s_{p,q}(S)=B^s_{p,q}(S)=  \cB^s_{p,q}(S)=\hat \cB^s_{p,q}(S)
\end{equation}
with equivalent norms.
\end{theorem}

\begin{proof} 
Let $S\subset\rn$ be a measurable set.
The first two inequalities in \eqref{besov rn} are obvious, since $\hat E_p(u,S,t)\le E_p(u,S,t)\le\omega(u,S,t)_p$ for all $t>0$. 
In order to show the last inequality, let $u\in N^s_{p,q}(S)$, $(g_k)_{k\in\z}\in\D^s(u)$ and $k\in\z$. Then
\[
\begin{split}
\omega(u,S,2^{-k})_p
&=\sup_{j\ge k}\sup_{2^{-j-1}\le |h|< 2^{-j}}\Big(\int_{S_h}|u(x+h)-u(x)|^p\,dx\Big)^{1/p}\\
&\le C\sup_{j\ge k}2^{-js}\Big(\int_{S_h}g_j(x+h)^p+g_j(x)^p\,dx\Big)^{1/p} \\
&\le C\sup_{j\ge k}2^{-js}\|g_j\|_{L^p(S)}
\le C\sum_{j\ge k}2^{-js}\|g_j\|_{L^p(S)}.
\end{split}
\]
If $0<q<\infty$, then by the estimate above and by Lemma \ref{summing lemma}, we obtain
\[
\begin{split}
\int_0^1\big(t^{-s}\omega(u,S,t)_p\big)^{q}\frac{dt}{t}
&\le C\sum_{k\ge 0}\big(2^{ks}\omega(u,S,2^{-k})_p\big)^{q}
\le C\sum_{k\ge 0}\Big(\sum_{j\ge k}2^{(k-j)sp}\|g_j\|_{L^p(S)}^p\Big)^{q/p}\\
&\le C\sum_{j\in\z}\|g_j\|_{L^p(S)}^q.
\end{split}
\]
In the case $q=\infty$, we have
\[
\begin{split}
\sup_{0<t<1}t^{-s}\omega(u,S,t)_p
&\le C\sup_{k\ge 0}2^{ks}\omega(u,S,2^{-k})_p\\
&\le C\sup_{k\ge 0}\Big(\sum_{j\ge k}2^{(k-j)sp}\|g_j\|_{L^p(S)}^p\Big)^{1/p}\\
&\le C\sup_{j\in\z}\|g_j\|_{L^p(S)}.
\end{split}
\]
The claim follows by taking the infimum over all $(g_k)_{k\in\z}\in\mathbb D^s(u)$.
\medskip

Next assume that $S$ satisfies the measure density condition.
Then $(S,d,\mu)$, where $d$ and $\mu$ are the restrictions of the Euclidean metric and the Lebesgue measure to $S$, 
satisfies the doubling property locally, that is, for a given $R>0$, there exists a constant $C=C(n,c_m,R)$ such that
\[
\mu(B(x,2r))\le C\mu(B(x,r))
\]
for all $x\in S$ and $0<r\le R$. Now, the inclusion $\hat \cB^s_{p,q}(S)\subset N^s_{p,q}(S)$ and, hence, \eqref{besov spaces 
coincide} follows essentially from the proof of \cite[Thm 2.1]{GKZ}.
\end{proof}

Theorem \ref{relation of besov spaces2} implies that if $\Omega$ is an extension domain for one of the spaces 
$B^s_{p,q}$, $\cB^s_{p,q}$, $\hat \cB^s_{p,q}$, then it is an extension domain for $N^s_{p,q}$. By combining
Theorems \ref{relation of besov spaces2},  \ref{measure density from extension} and \ref{m extension}, we obtain
the first main result of this section.

\begin{theorem}\label{ext B measure density rn}
Let $0<s<1$, $0<p<\infty$ and $0<q\le\infty$.
Then $\Omega\subset \rn$ is an extension domain for  $B^s_{p,q}$  (resp. for $\cB^s_{p,q}$  or for $\hat \cB^s_{p,q}$) if and 
only if it satisfies measure density condition \eqref{measure density}.
\end{theorem}
\begin{remark}
The definition of the space $B^s_{p,q}(S)$ depends on the translation structure of $\rn$, but
the definitions of $\cB^s_{p,q}(S)$  and $\hat\cB^s_{p,q}(S)$ can be naturally extended to the metric setting.
With minor modifications in the proofs, we obtain the following counterparts of 
Theorems \ref{relation of besov spaces2} and \ref{ext B measure density rn}. We leave the details to the reader.
\end{remark}

\begin{theorem}\label{relation of besov spaces} 
Let $X$ be a doubling metric measure space.
Let $0<s<1$, $0<p<\infty$ and $0<q\le\infty$. 
Then
\[
N^s_{p,q}(S)\subset  \cB^s_{p,q}(S)\subset\hat \cB^s_{p,q}(S)
\]
and there is a constant $C>0$ such that
\[
\|\cdot\|_{\hat \cB^s_{p,q}(S)}\le \|\cdot\|_{\cB^s_{p,q}(S)}
\le  C\|\cdot\|_{N^s_{p,q}(S)},
\]
for each measurable set $S\subset X$.

If $S$ satisfies measure density condition \eqref{measure density}, then
\[
N^s_{p,q}(S)= \cB^s_{p,q}(S)=\hat \cB^s_{p,q}(S)
\]
with equivalent norms.
\end{theorem}

\begin{theorem}\label{ext B measure density}
Let $X$ be a $Q$-regular, geodesic metric measure space.
Let $0<s<1$, $0<p<\infty$ and $0<q\le\infty$.
Then $\Omega\subset X$ is an extension domain for  $\cB^s_{p,q}$  (resp. for $\hat \cB^s_{p,q}$) 
if and only if it satisfies measure density condition \eqref{measure density}.
\end{theorem}

\subsection{Triebel--Lizorkin spaces on subsets of $\rn$}

Let $S\subset \rn$ be a measurable set.
Let $0<s<1$, $0<p,q<\infty$ and $0<r<\min\{p,q\}$.

The Triebel--Lizorkin space ${\cF}^s_{p,q}(S)$ consists of functions $u\in L^p(S)$, for which the norm
\[
\|u\|_{{\cF}^s_{p,q}(S)}
=\|u\|_{L^p(S)}+\|g\|_{L^p(S)},
\]
where
\[
\begin{split}
g(x)
=&\bigg(\int_0^1\bigg(t^{-s}\bigg(\frac1{|B(x,t)|}\int_{B(x,t)\cap S}
|u(y)-u(x)|^r\,dy\bigg)^{1/r}\bigg)^q\,\frac{dt}{t}\bigg)^{1/q},
\end{split}
\]
is finite. For $S=\rn$, this definition coincides with the classical difference definition.

The Triebel--Lizorkin space ${\hat\cF}^s_{p,q}(S)$ consists of functions $u\in L^p(S)$, for which the norm
\[
\|u\|_{{\hat\cF}^s_{p,q}(S)}
=\|u\|_{L^p(S)}+\|\hat g\|_{L^p(S)},
\]
where
\[
\hat g(x)
=\bigg(\int_0^1 \bigg(t^{-s}\bigg(\frac1{|B(x,t)|}\inf_{c\in\re}\int_{B(x,t)\cap S}|u(y)-c|^r\,dy\bigg)^{1/r}
\bigg)^q\,\frac{dt}{t}\bigg)^{1/q},
\]
is finite. This definition, with $r=1$, $p,q>1$, was used in \cite{ShRn}.
\begin{remark}
If in the definitions above we integrate over $(0,\infty)$ instead of $(0,1)$, we end up with the equivalent norms.
\end{remark}

Corresponding to Theorem \ref{relation of besov spaces2} for Besov spaces, we have the following result.

\begin{theorem}\label{relation of TL spaces}
Let $0<s<1$, $0<p,q<\infty$. Then
\begin{equation}
M^s_{p,q}(S)\subset \cF^s_{p,q}(S)\subset \hat \cF^s_{p,q}(S)
\end{equation}
and there is a constant $C>0$ such that
\begin{equation}\label{relation of TL norms}
\|\cdot\|_{\hat \cF^s_{p,q}(S)}\le \|\cdot\|_{\cF^s_{p,q}(S)}\le C\|\cdot\|_{M^s_{p,q}(S)},
\end{equation}
for each measurable set $S\subset \rn$. 

If $S$ satisfies measure density condition \eqref{measure density}, then
\begin{equation} \label{TL spaces coincide}
M^s_{p,q}(S)= \cF^s_{p,q}(S)=\hat \cF^s_{p,q}(S)
\end{equation}
with equivalent norms.
\end{theorem}

\begin{proof}
Let $S\subset \rn$ be a measurable set. The first inequality in \eqref{relation of TL norms} is obvious because
$\hat g(x)\le g(x)$. To prove the second inequality,
let $u\in M^s_{p,q}(S)$, $(g_k)_{k\in\z}\in\D^s(u)$ and $k\in\z$. Then, for almost every $x\in S$,
\[
\begin{split}
\frac{1}{|B(x,2^{-k})|}&\int_{B(x,2^{-k})\cap S}|u(x)-u(y)|^r\,dy\\
&=\sum_{j\ge k}\frac{1}{|B(x,2^{-k})|}\int_{(B(x,2^{-j})\setminus B(x,2^{-j-1}))\cap S}|u(x)-u(y)|^r\,dy\\
&\le C\sum_{j\ge k}2^{-jsr}\Big( g_j(x)^r+\vint{B(x,2^{-j})} (g_j(y)^r\ch S(y))\,dy\Big)\\
&\le C\sum_{j\ge k}2^{-jsr}\M(g_j^r\ch S)(x),
\end{split}
\]
and, hence,
\[
\begin{split}
g(x)^q
&=\int_0^1 \bigg(t^{-s}\bigg(\frac{1}{|B(x,t)|}\int_{B(x,t)\cap S}|u(x)-u(y)|^r\,dy\bigg)^{1/r}\bigg)^q\,\frac{dt}{t}\\
&\le C\sum_{k\ge 0} \bigg( 2^{ks}\bigg(\frac{1}{|B(x,2^{-k})|}\int_{B(x,2^{-k})\cap S}
   |u(x)-u(y)|^r\,dy\bigg)^{1/r}\bigg)^{q}\\
&\le C\sum_{k\ge 0}\Big(\sum_{j\ge k}2^{(k-j)sr}\M(g_j^r\ch S)(x)\Big)^{q/r}\\
&\le C\sum_{j\in\z}\big( \M(g_j^r\ch S)(x) \big)^{q/r},
\end{split}
\]
where the last inequality follows from Lemma \ref{summing lemma}.
By the Fefferman--Stein vector valued maximal function theorem,
we obtain
\begin{align*}
\|g\|_{L^p(S)}
&\le C\|(\M (g_k^r\ch S))_{k\in\z}\|_{L^{p/r}(\rn,\,l^{q/r})}^{1/r}
\le C\|(g_k^r\ch S)_{k\in\z}\|_{L^{p/r}(\rn,\,l^{q/r})}^{1/r}\\
&=C\|(g_k)_{k\in\z}\|_{L^{p}(S,l^{q})}.
\end{align*}
The claim follows by taking the infimum over all $(g_k)_{k\in\z}\in\D^s(u)$.

If $S$ satisfies measure density condition \eqref{measure density}, then \eqref{TL spaces coincide} follows from the proof of 
\cite[Thm 3.1]{GKZ}.
\end{proof}

By combining Theorems \ref{relation of TL spaces},  \ref{measure density from extension} and \ref{m extension},
we obtain the second main result of this section.

\begin{theorem}\label{ext TL measure density}
Let $0<s<1$, $0<p,q<\infty$.
Then $\Omega\subset \rn$ is an extension domain for $\cF^s_{p,q}$ (resp. for  $\hat \cF^s_{p,q}$) if and only if it satisfies 
measure density condition \eqref{measure density}.
\end{theorem}

\begin{remark}
As in the case of Besov spaces, the definitions of Triebel--Lizorkin spaces and the results above have counterparts in metric 
setting.
\end{remark}

\begin{remark}\label{rm:IntrinsicDefinitionsTriebel-Lizorkin}
For domains $\Omega\subset\rn$, Triebel--Lizorkin spaces ${C}^s_{p,q}(\Omega)$ consisting of functions $u\in L^p(\Omega)
$, for which the norm
\[
\|u\|_{{C}^s_{p,q}(\Omega)}
=\|u\|_{L^p(\Omega)}+\|h\|_{L^p(\Omega)},
\]
where
\begin{equation}
h(x)
=\bigg(\int_0^{\tau\delta(x)} \bigg(t^{-s}\bigg(\inf_{c\in\re}\ \vint{B(x,t)}|u(y)-c|^r\,dy\bigg)^{1/r}
\bigg)^q\,\frac{dt}{t}\bigg)^{1/q},
\end{equation}
$0<\tau<1$ and $\delta(x)=d(x,\Omega^c)$, is finite have been studied in \cite{Se} and \cite{Mi}. 
Since
\[
M^s_{p,q}(\Omega)\subset \hat\cF^s_{p,q}(\Omega)\subset C^s_{p,q}(\Omega),
\]
Theorem \ref{measure density from extension} implies that $C^s_{p,q}$-extension domains are regular. 
The converse is not true. For example, the slit disc $\Omega=B(0,1)\setminus([0,1)\times \{0\} )\subset\re^2$ is regular, but it 
is clearly not a $C^s_{p,q}$-extension domain.
\end{remark}

On the other hand, if $\Omega\subset\rn$ is an $(\eps,\delta)$-domain, then
$C^s_{p,q}(\Omega)$ coincides with the other Triebel--Lizorkin spaces considered in this section. 
This follows, for example, from the extension results obtained in \cite{Se} and \cite{Mi}, and the characterization of extension 
domains for the spaces $C^s_{p,q}$ below.
\begin{theorem}\label{ext Cspq spaces measure density}
Let $0<s<1$, $0<p,q<\infty$.
Then $\Omega\subset \rn$ is an extension domain for $C^s_{p,q}$ if and only if it satisfies measure density condition 
\eqref{measure density} and $C^s_{p,q}(\Omega)=M^s_{p,q}(\Omega)$.
\end{theorem}

\medskip
\noindent {\bf Acknowledgements:} 
The research was supported by the Academy of Finland, grants no.\ 135561, 141021 and 272886.
Part of the paper was written when the third author was visiting the Universit\'e Paris-Sud (Orsay) in Springs 2013 and 2014 
and while the authors visited The Institut Mittag--Leffler in Fall 2013. They thank these institutions for their kind hospitality.

\vspace{0.5cm}
\noindent
\small{\textsc{T.H.},}
\small{\textsc{Department of Mathematics},}
\small{\textsc{P.O. Box 11100},}
\small{\textsc{FI-00076 Aalto University},}
\small{\textsc{Finland}}\\
\footnotesize{\texttt{toni.heikkinen@aalto.fi}}

\vspace{0.3cm}
\noindent
\small{\textsc{L.I.},}
\small{\textsc{Department of Mathematics and Statistics},}
\small{\textsc{P.O. Box 35},}
\small{\textsc{FI-40014 University of Jyv\"askyl\"a},}
\small{\textsc{Finland}}\\
\footnotesize{\texttt{lizaveta.ihnatsyeva@aalto.fi}}

\vspace{0.3cm}
\noindent
\small{\textsc{H.T.},}
\small{\textsc{Department of Mathematics and Statistics},}
\small{\textsc{P.O. Box 35},}
\small{\textsc{FI-40014 University of Jyv\"askyl\"a},}
\small{\textsc{Finland}}\\
\footnotesize{\texttt{heli.m.tuominen@jyu.fi}}\\
\small{\textsc{+358 40 805 4594}}

\end{document}